\documentclass[a4paper,leqno]{amsart}

\usepackage{amsmath}
\usepackage{amsthm}
\usepackage{amssymb}
\usepackage{amsfonts}
\usepackage[applemac]{inputenc}
\usepackage{mathrsfs}
\usepackage{pdfsync}
\usepackage[all]{xy}
\usepackage{graphicx}
\usepackage{epstopdf}
\usepackage{epsfig}

\def\C{{\mathbb C}}
\def\R{{\mathbb R}}
\def\N{{\mathbb N}}

\def\le{\leqslant}
\def\ge{\geqslant}

\newcommand{\eps}{\varepsilon}
\newcommand{\e}{\varepsilon}

\DeclareMathOperator{\diver}{div}

\theoremstyle{plain}
\newtheorem{theorem}{Theorem}[section]
\newtheorem{lemma}[theorem]{Lemma}

\newtheorem{proposition}[theorem]{Proposition}

\theoremstyle{definition}

\newtheorem{remark}[theorem]{Remark}
\newtheorem*{remark*}{Remark}

\numberwithin{equation}{section}

\begin{document}

\title[Classical limit of a TDSCF system]
{On the classical limit of a time-dependent self-consistent field system: analysis and computation}
\author[S. Jin]{Shi Jin}
\author[C. Sparber]{Christof Sparber}
\author[Z. Zhou]{Zhennan Zhou}

\address[S. Jin]{Department of Mathematics, Institute of Natural Sciences and MOE Key Lab in Scientific and Engineering 
Computing, Shanghai Jiao Tong University, 800 Dong Chuan Road, Shanghai 200240, China \& Department of Mathematics, University of Wisconsin-Madison, 480 Lincoln Drive, Madison, WI 53706, USA.}
\email{jin@math.wisc.edu}
\address[Z. Zhou]{Department of Mathematics, University of Wisconsin-Madison, 480 Lincoln Drive, Madison, WI 53706, USA.}
\email{zhou@math.wisc.edu}
\address[C.~Sparber]
{Department of Mathematics, Statistics, and Computer Science, M/C 249, University of Illinois at Chicago, 851 S. Morgan Street, Chicago, IL 60607, USA.}
\email{sparber@math.uic.edu}

\begin{abstract}
We consider a coupled system of Schr\"odinger equations, arising in quantum mechanics via the so-called time-dependent self-consistent field method. 
Using Wigner transformation techniques we study the corresponding classical limit dynamics in two cases. In the first case, the classical limit is only taken 
in one of the two equations, leading to a mixed quantum-classical model which is closely connected to the well-known Ehrenfest method in molecular dynamics. 
In the second case, the classical limit of the full system is rigorously established, resulting in a system of coupled Vlasov-type equations. 
In the second part of our work, we provide a numerical study of the coupled semi-classically scaled Schr\"odinger equations and of the mixed quantum-classical 
model obtained via Ehrenfest's method. 
A second order (in time) method is introduced for each case. We show that the proposed methods allow time steps independent of the 
semi-classical parameter(s) while still capturing the correct behavior of physical observables. It also becomes clear that the order of accuracy of our methods 
can be improved in a straightforward way.
\end{abstract}

\date{\today}


\thanks{This work was partially supported by NSF grants DMS-1114546 and DMS-1107291: NSF Research
Network in Mathematical Sciences KI-Net: Kinetic description of emerging challenges in multiscale problems
of natural sciences. C.S. acknowledges support by the NSF through grant numbers DMS-1161580 and DMS-1348092}

\maketitle


\section{Introduction}\label{sec:intro}

The numerical simulation of many chemical, physical, and biochemical phenomena requires the direct simulation of dynamical processes within  
large systems involving quantum mechanical effects. However, if the entire system is treated quantum mechanically, the numerical simulations are 
often restricted to relatively small model problems on short time scales due to the formidable computational cost. In order to overcome this difficulty, 
a basic idea is to separate the involved degrees of freedom into two different categories: one, which involves variables that behave effectively classical (i.e., 
evolving on slow time- and large spatial scales) and one which encapsulates the (fast) quantum mechanical dynamics within a certain portion of the full system. 
For example, for a system consisting of many molecules, one might designate the electrons as the fast degrees of freedom and the atomic nuclei as the slow
degrees of freedom. 

Whereas separation of the whole system into a classical part and a quantum mechanical part is certainly not an easy task, it is, by now, widely studied 
in the physics literature and often leads to what is called {\it time-dependent self-consistent field equations} (TDSCF), see, e.g., \cite{BNS, Dr, H, ZNN, SuMi} 
and the references therein. In the TDSCF method, one typically assumes that the total wave function of the system $\Psi(t, X)$, with $X=(x,y)$, 
can be approximated by 
\[
\Psi(X, t) \approx \psi(x,t) \varphi(y,t), 
\]
where $x$ and $y$ denote the degrees of freedom within a certain subsystem, only. The precise nature of this approximation thereby strongly depends on 
the concrete problem at hand. Disregarding this issue for the moment, one might then, in a second step, hope to derive a self-consistently coupled system for $\psi$ and $\varphi$ and 
approximate it, at least partially, by the associated classical dynamics.

In this article we will study a simple model problem for such a TDSCF system, motivated by \cite{Dr}, but 
one expects that our findings extend to other self-consistent models as well.
We will be interested in deriving various (semi-)classical approximations to 
the considered TDSCF system, resulting in either a mixed quantum-classical model, or a fully classical model. 
As we shall see, this also gives a rigorous justification of what is known as the {\it Ehrenfest method} in the physics literature, cf. \cite{BNS, Dr}.
To this end, we shall be heavily relying on {\it Wigner transformation} 
techniques, developed in \cite{GMMP, LiPa}, which have been proved to be superior in many aspects to the more classical WKB approximations, 
see, e.g. \cite{SMM} for a broader discussion. One should note that the use of Wigner methods to study the classical limit of nonlinear (self-consistent) quantum mechanical models 
is not straightforward and usually requires additional assumptions on the quantum state, cf. \cite{MaMa, LiPa}. It turns out that in our case we can get by without them.

In the second part of this article we shall then be interested in designing an efficient and accurate numerical method which 
allows us to pass to the classical limit in the TDSCF system within our numerical algorithm. We will be particularly interested in the  
meshing strategy required to accurately approximate the wave functions, or to capture the correct physical observables (which are quadratic quantities of the wave function). 
To this end, we propose a second order (in time) method based on an operator splitting and a spectral approximation of the TDSCF equations as well as the obtained Ehrenfest model. 
These types of methods have been proven to be very effective in earlier numerical studies, see \cite{NLS, Kl, TS, SL-TS} for previous results and \cite{JMS} 
for a review of the current state-of-the art of numerical methods for semi-classical Schr\"odinger type models. In comparison to 
the case of a single (semi-classical) nonlinear Schr\"odinger equation with power law nonlinearities, where on has to use time steps 
which are comparable to the  size of the small semi-classical parameter (see \cite{NLS}), it turns out that in our case, despite of the nonlinearity, 
we can rigorously justify that one can take time steps {\it independent} of the semi-classical parameter and still capture the correct classical limit of physical observables. 

The rest of this paper is now organized as follows: In Section \ref{sec:prelim}, we present the considered TDSCF system and discuss some of its basic mathematical properties, 
which will be used later on. In Section \ref{sec:Wigner}, a brief introduction to the Wigner transforms and Wigner measures is given.
In Section \ref{sec:semi} we study the semi-classical limit, resulting in a mixed quantum-classical limit system. In Section \ref{sec:classical} the completely classical approximation 
of the TDSFC system is studied by means of two different limiting processes, both of which result in the same classical model.
The numerical methods used for the TDSCF equations and the Ehrenfest equations are then introduced in Section \ref{sec:spectral}. Finally, 
we study several numerical tests cases in Section \ref{sec:tests} in order to verify the properties of our methods.


\section{The TDSCF system}\label{sec:prelim} 

\subsection{Basic set-up and properties} In the following, we take $x\in \R^d$, $y\in \R^n$, with $d,n\in \N$, and 
denote by $\langle \cdot, \cdot\rangle_{L_x^{2}}$ and $\langle \cdot, \cdot\rangle_{L_y^{2}}$ the usual inner product in 
$L^2(\R^d_x)$ and $L^2(\R^n_y)$, respectively, i.e. 
\[
\left\langle f, g \right\rangle_{L_z^{2}} \equiv \int_{\R^m} \bar{f}(z)g(z)dz.
\]
The total Hamiltonian of the system acting on $L^2(\R^{d+n})$ is assumed to be of the form
\begin{equation}\label{eq:ham}
H= - \frac{\delta^2}{2} \Delta_x - \frac{\eps^2}{2} \Delta_y + V(x,y) ,  
\end{equation}
where $V(x,y)\in \R$ is some (time-independent) real-valued potential. Typically, one has
\begin{equation}\label{eq:pot}
V(x,y) = V_1(x)+ V_2(y) + W(x,y),
\end{equation} where $V_{1,2}$ are external potentials acting only on the respective subsystem 
and $W$ represents an internal coupling potential in between the two subsystems. From now on, we shall assume that 
$V$ is smooth and rapidly decaying at infinity, i.e.
\begin{equation}\label{eq:hyp}
V\in \mathcal S(\R^d_x\times \R^n_y).
\end{equation}
In particular, we may, without restriction of generality, assume $V(x,y)\ge 0$. 

\begin{remark} The proofs given below will show that, with some effort, all of our results can be extended to the case of 
potentials $V\in C^2_{\rm b}(\R^d_x\times \R^n_y)$.
\end{remark}

In \eqref{eq:ham}, the Hamiltonian is already written in dimensionless form, such that only two (small) parameters $\eps, \delta >0$ remain. 
In the following, they play the role of dimensionless Planck's constants. Dependence with respect to these parameters will be denoted by 
superscripts. The TDSCF system at hand is then (formally given by \cite{Dr}) the following system of self-consistently coupled Schr\"odinger equations 
\begin{equation}\label{eq:TDSCF}
\left\{ \begin{split}
i \delta \partial_t \psi^{\eps, \delta}  = \left (-\frac{\delta^2}{2}\Delta_x  + \langle \varphi^{\eps, \delta},V\varphi^{\eps, \delta} \rangle_{L_y^{2}} \right) \psi^{\eps, \delta} \, ,
\quad \psi^{\eps, \delta} _{\mid t=0} = \psi^\delta_{\rm in}(x),\\
i\varepsilon \partial_t \varphi^{\eps, \delta}  = \left(-\frac{\varepsilon^2}{2}\Delta_y  + \langle \psi^{\eps, \delta},h^\delta \psi^{\eps, \delta} \rangle_{L_x^{2}}\right) \varphi^{\eps, \delta}  \, ,
\quad \varphi^{\eps, \delta}_{\mid t=0} = \varphi^\eps_{\rm in}(y),
\end{split}\right.
\end{equation}
where we denote by
\begin{equation}\label{eq:subham}
h^\delta=-\frac{\delta^2}{2}\Delta_x+V(x,y),
\end{equation}
the Hamiltonian of the subsystem represented by the $x$-variables (considered as the purely quantum mechanical variables) and in which $y$ only enters as a parameter. 
For simplicity, we assume that at $t=0$ the data $ \psi_{\rm in}^{\delta}$ only depends on $\delta$, and that 
$ \varphi_{\rm in}^{\eps}$ only depends on $\eps$, which 
means that the simultaneous dependence on both parameters is only induced by the time-evolution.
Finally, the coupling terms are explicitly given by
\[
 \langle \varphi^{\eps, \delta},V\varphi^{\eps, \delta} \rangle_{L_y^{2}} = \int_{\R_y^n} V(x,y) |\varphi^{\eps, \delta}(y,t)|^2 \, dy =: \Upsilon^{\eps, \delta}(x,t) ,
\]
and after formally integrating by parts
\[
\langle \psi^{\eps, \delta},h^\delta \psi^{\eps, \delta} \rangle_{L_x^{2}} =  \int_{\R_x^d} \frac{\delta^2}{2} | \nabla \psi^{\eps, \delta}(x,t)|^2 + V(x,y) |\psi^{\eps, \delta}(x,t)|^2 \, dx =:\Lambda^{\eps, \delta}(y,t).
\]
Throughout this work we will always interpret the term $ \langle \psi^{\eps, \delta},h^\delta \psi^{\eps, \delta} \rangle_{L_x^{2}}$ as above, i.e., in the weak sense. 
Both $\Upsilon^{\eps, \delta}$ and $\Lambda^{\eps, \delta}$ are time-dependent, real-valued potentials, computed self-consistently via the dynamics of $\varphi^{\eps, \delta}$ and 
$\psi^{\eps, \delta}$, respectively. Note that 
\begin{equation}\label{eq:Lambda}
\Lambda^{\eps, \delta}(y,t)= \frac{\delta^2}{2} \| \nabla  \psi^{\eps, \delta} \|^2_{L^2_x}+ \langle \psi^{\eps, \delta},V\psi^{\eps, \delta} \rangle_{L_x^{2}} 
\equiv \vartheta^{\eps, \delta}(t) + \langle \psi^{\eps, \delta},V\psi^{\eps, \delta} \rangle_{L_x^{2}}.
\end{equation}
Here, the purely time-dependent part $ \vartheta^{\eps, \delta}(t)$, could in principle be absorbed into the definition of 
$\varphi^{\eps, \delta}$ via a Gauge transformation, i.e.,
\begin{equation}\label{eq:gauge}
\varphi^{\eps, \delta} (x,t) \mapsto \widetilde\varphi^{\eps, \delta} (x,t):= {\varphi^{\eps, \delta}}(x,t) \exp\left(-\frac{i}{\eps} \int_0^t\vartheta^{\eps, \delta}(s)\, ds\right),
\end{equation}
For the sake of simplicity, we shall refrain from doing so, but this nevertheless shows that the two coupling terms are in essence of the same form. 
Also note that this Gauge transform leaves any $H^s(\R^d)$-norm of $\varphi^{\eps, \delta}$ invariant (but clearly depends on the solution of the second equation within the TDSCF system).

\begin{remark} 
For potentials of the form \eqref{eq:pot}, one can check that in the case where $W(x,y)\equiv 0$, i.e., no coupling term, one can use 
similar gauge transformations to completely decouple the two equations in \eqref{eq:TDSCF} and obtain two linear Schr\"odinger equations in 
$x$ and $y$, respectively.
\end{remark}

An important physical quantity is the total masses of the system, 
\begin{equation}\label{mass}
M^{\eps, \delta}(t): = \| \psi^{\eps, \delta}(\cdot, t)\|^2_{L_x^{2}} +  \| \varphi^{\eps, \delta}(\cdot, t)\|^2_{L_y^{2}}\equiv m_1^{\eps, \delta}(t)+m_2^{\eps, \delta}(t).
\end{equation}
where $m_1^{\eps, \delta}$, $m_2^{\eps, \delta}$ denote the masses of the respective subsystem. One can then prove that these are 
conserved by the time-evolution of \eqref{eq:TDSCF}. 
\begin{lemma}\label{lem:mass}
Assume that $\psi^{\eps, \delta}\in C(\R_t; H^1(\R^d_x))$ and $\varphi^{\eps, \delta}\in C(\R_t; L^2(\R^n_y))$ solve \eqref{eq:TDSCF}, then
\[
m_1^{\eps, \delta}(t) = m_1^{\eps, \delta}(0), \quad m_2^{\eps, \delta}(t) = m_2^{\eps, \delta}(0), \quad \forall\, t\in \R.
\]
\end{lemma}
\begin{proof}
Assuming for the moment, that both $\psi^{\eps, \delta}$ and $\varphi^{\eps, \delta}$ are sufficiently smooth and decaying, we 
multiply the first equation in \eqref{eq:TDSCF} with $\overline {\psi^{\eps, \delta}}$ and formally integrate with respect to $x\in \R^{d}_x$. Taking the real part of the resulting 
expression and having in mind that $\Upsilon^{\eps, \delta}(y,t)\in \R$, yields
\[
\frac{d}{dt} m_1^{\eps, \delta}(t) \equiv \frac{d}{dt} \| \psi^{\eps, \delta}(\cdot, t)\|_{L^2_x}^2 = 0,
\]
which, after another integration in time, is the desired result for $m_1^{\eps, \delta}(t)$. By the same argument one can show the result for $m_2^{\eps, \delta}(t)$. Integration in time 
in combination with a density argument then allows to 
extend the result to more general solutions in $H^1$ and $L^2$, respectively.
\end{proof}

We shall, from now on assume that the initial data is normalized such that $m_1^{\eps, \delta}(0)=m_2^{\eps, \delta}(0)=1$. Using this normalization, the total energy 
of the system can be written as
\begin{equation}\label{energy}
\begin{split}
E^{\eps, \delta}(t) :=  &\, \frac{\delta^2}{2} \| \nabla  \psi^{\eps, \delta}(\cdot, t) \|^2_{L^2_x} + \frac{\eps^2}{2} \| \nabla  \varphi^{\eps, \delta}(\cdot, t) \|^2_{L^2_y} \\
&\,  + 
\iint_{\R^{d+n} } V(x,y)   |\psi^{\eps, \delta}(x,t)|^2 |\varphi^{\eps, \delta}(y,t)|^2\, dx\, dy.
\end{split}
\end{equation}
\begin{lemma}\label{lem:energy}
Assume that $\psi^{\eps, \delta}\in C(\R_t; H^1(\R^d_x))$ and $\varphi^{\eps, \delta}\in C(\R_t; H^1(\R^n_y))$ solve \eqref{eq:TDSCF}, then
\[
E^{\eps, \delta}(t) = E^{\eps, \delta}(0), \quad \forall\, t\in \R.
\]
\end{lemma}

\begin{proof}
Assuming, as before that $\psi^{\eps, \delta}$ and $\varphi^{\eps, \delta}$ are sufficiently regular (and decaying), 
the proof is a lengthy but straightforward calculation. More precisely, using the shorthand
\[
E^{\varepsilon,\delta}(t)=\frac{\delta^2}{2} \| \nabla  \psi^{\eps, \delta}(\cdot, t) \|^2_{L^2_x} + \frac{\eps^2}{2} \| \nabla  \varphi^{\eps, \delta}(\cdot, t) \|^2_{L^2_y} + \langle \psi^{\varepsilon,\delta}\varphi^{\varepsilon,\delta}, V \psi^{\varepsilon,\delta}\varphi^{\varepsilon,\delta} \rangle_{L^2_{x,y}},
\]
one finds that
\[
\frac{d}{dt} E^{\varepsilon,\delta}(t)=(\rm{I})+(\rm{II})+(\rm{III})+(\rm{IV}),
\]
where we denote
\begin{align*}
({\rm I}):=&\frac{\delta^2}{2}\langle \nabla_x \partial_t \psi^{\varepsilon,\delta},\nabla_x \psi^{\varepsilon,\delta}\rangle_{L^2_x}  +\frac{\delta^2}{2}\langle \nabla_x  \psi^{\varepsilon,\delta},\nabla_x \partial_t \psi^{\varepsilon,\delta}\rangle_{L^2_x} ,\\
({\rm II}):=&\frac{\varepsilon^2}{2}\langle \nabla_y \partial_t \varphi^{\varepsilon,\delta},\nabla_y \varphi^{\varepsilon,\delta}\rangle_{L^2_y}  +\frac{\varepsilon^2}{2}\langle \nabla_y  \varphi^{\varepsilon,\delta},\nabla_y \partial_t \varphi^{\varepsilon,\delta}\rangle_{L^2_y},\\
({\rm III}):=&\langle \partial_t \psi^{\varepsilon,\delta} \varphi^{\varepsilon,\delta},V\psi^{\varepsilon,\delta}\varphi^{\varepsilon,\delta}\rangle_{L^2_{x,y}}+ \langle \psi^{\varepsilon,\delta} \varphi^{\varepsilon,\delta},V \partial_t \psi^{\varepsilon,\delta}\varphi^{\varepsilon,\delta}\rangle_{L^2_{x,y}},\\
({\rm IV}):=&\langle \psi^{\varepsilon,\delta} \partial_t \varphi^{\varepsilon,\delta},V\psi^{\varepsilon,\delta}\varphi^{\varepsilon,\delta}\rangle_{L^2_{x,y}}+ \langle \psi^{\varepsilon,\delta} \varphi^{\varepsilon,\delta},V \psi^{\varepsilon,\delta} \partial_t\varphi^{\varepsilon,\delta}\rangle_{L^2_{x,y}}.
\end{align*}
We will now show that $(\rm{I})+(\rm{III})=0$. By using the \eqref{eq:TDSCF}, one gets
\begin{eqnarray*}
({\rm I}) & = & -\frac{\delta}{2i} \left\langle \nabla_x  \left(-\frac{\delta^2}{2}\Delta_x + \langle \varphi^{\varepsilon,\delta},V\varphi^{\varepsilon,\delta} \rangle_{L^2_y} \right) \psi^{\varepsilon,\delta},\nabla_x \psi^{\varepsilon,\delta}\right\rangle_{L^2_x}  \\
&  & +\frac{\delta}{2i} \left\langle \nabla_x  \psi^{\varepsilon,\delta},\nabla_x  \left(-\frac{\delta^2}{2}\Delta_x + \langle \varphi^{\varepsilon,\delta},V\varphi^{\varepsilon,\delta} \rangle_{L^2_y}  \right) \psi^{\varepsilon,\delta}\right\rangle_{L^2_x}  \\
& = & -\frac{\delta}{2i} \langle (\nabla_x   \langle \varphi^{\varepsilon,\delta},V\varphi^{\varepsilon,\delta} \rangle_{L^2_y} )\  \psi^{\varepsilon,\delta},\nabla_x \psi^{\varepsilon,\delta}\rangle_{L^2_x}  \\
&  & +\frac{\delta}{2i} \langle \nabla_x  \psi^{\varepsilon,\delta},(\nabla_x   \langle \varphi^{\varepsilon,\delta},V\varphi^{\varepsilon,\delta} \rangle_{L^2_y})  \  \psi^{\varepsilon,\delta}\rangle_{L^2_x}  \\
& = &  \frac{\delta}{2i} \left \langle \psi^{\varepsilon,\delta} ,\left[\langle \varphi^{\varepsilon,\delta},V\varphi^{\varepsilon,\delta} \rangle_{L^2_y}  ,\Delta_x\right]\psi^{\varepsilon,\delta} \right\rangle_{L^2_x} ,
\end{eqnarray*}
where $[A,B]:=AB-BA$ denotes the commutator bracket. Similarly, one finds that
\begin{eqnarray*}
({\rm III}) & = &-\frac{\delta}{2i} \left\langle \left(-\frac{\delta^2}{2}\Delta_x + \langle \varphi^{\varepsilon,\delta},V\varphi^{\varepsilon,\delta} \rangle_{L^2_y}  \right) \psi^{\varepsilon,\delta} \phi^{\varepsilon,\delta} ,V\psi^{\varepsilon,\delta} \phi^{\varepsilon,\delta}  \right\rangle_{L^2_{x,y}} \\
 &  &+\frac{\delta}{2i} \left\langle  \psi^{\varepsilon,\delta} \phi^{\varepsilon,\delta} ,V \left(-\frac{\delta^2}{2}\Delta_x + \langle \varphi^{\varepsilon,\delta},V\varphi^{\varepsilon,\delta} \rangle_y \right) \psi^{\varepsilon,\delta} \phi^{\varepsilon,\delta}  \right\rangle_{L^2_{x,y}}\\
   & = & \frac{\delta}{2i} \left\langle \Delta_x\psi^{\varepsilon,\delta}, \langle \varphi^{\varepsilon,\delta},V\varphi^{\varepsilon,\delta} \rangle_y\psi^{\varepsilon,\delta} \right\rangle_{L^2_x} \\
  &  & -\frac{\delta}{2i} \left\langle \psi^{\varepsilon,\delta}, \langle \varphi^{\varepsilon,\delta},V\varphi^{\varepsilon,\delta} \rangle_y \Delta_x \psi^{\varepsilon,\delta} \right\rangle_{L^2_x} \\
  & = &  \frac{\delta}{2i} \left \langle \psi^{\varepsilon,\delta} ,\left[\Delta_x,\langle \varphi^{\varepsilon,\delta},V\varphi^{\varepsilon,\delta} \rangle_y \right]\psi^{\varepsilon,\delta} \right\rangle_{L^2_x} = - ({\rm I)},
\end{eqnarray*}
due to the fact that $[A, B] = -[B,A]$. Therefore, one concludes $(\rm{I})+(\rm{III})=0$. 
Analogously, one can show that $(\rm{II})+(\rm{IV})=0$ and hence, an integration in time yields $E^{\varepsilon,\delta}(t)=E^{\varepsilon,\delta}(0)$. 
Using a density arguments allows to extend this result to more general solution in $H^1$.
\end{proof}
Note however, that the energies defined for the respective subsystems are in general not conserved, unless $V(x,y)=V_1(x)+V_2(y)$.

\subsection{Existence of solutions} In this subsection, we shall establish global in-time existence of solutions to the TDSCF system \eqref{eq:TDSCF}. 
Since the dependence on $\eps$ and $\delta$ does not play a role here, we shall suppress their appearance for the sake of notation.

\begin{proposition}\label{prop:existence}
Let $V\in \mathcal S(\R^d_x\times \R^n_y)$ and $\psi_{\rm in} \in H^1(\R^d_x)$, $\varphi_{\rm in} \in H^1(\R^n_y)$. Then there exists a global 
strong solution $(\psi, \varphi) \in C(\R_t; H^1(\R^{d+n}))$ of \eqref{eq:TDSCF}, satisfying the conservation laws for mass and energy, as stated above.
\end{proposition}

Clearly, this also yields global existence for the system \eqref{eq:TDSCF} with $0< \eps, \delta <1$ included.

\begin{proof}
We shall first prove local (in-time) well-posedness of the initial value problem \eqref{eq:TDSCF}: To this end, we consider $\Psi(\cdot, t) =(\psi (\cdot, t), \varphi(\cdot, t)): \R^{d+n}\to \C^2$ and 
define the associated $L^2(\R^{d+n})$ norm by
\[
\| \Psi (\cdot, t) \|^2_{L^2}:= \| \psi (\cdot, t) \|^2_{L^2_x} + \| \varphi (\cdot, t) \|^2_{L^2_y},
\]
and consequently set $H^1(\R^{d+n}):= \{ \Psi \in L^2(\R^{d+n}) \, : \, |\nabla \Psi | \in L^2(\R^{d+n})\}$. Using this notation, the TDSCF system \eqref{eq:TDSCF} can be written as
\[
i \partial_t \Psi = \mathbb H \Psi + f(\Psi),
\]
where
\[
\mathbb H := \left ( \begin{matrix} -\frac{1}{2} \Delta_x & 0 \\ 0  & - \frac{1}{2} \Delta_y \end{matrix} \right), 
\quad f(\Psi) :=  \left ( \begin{matrix} \langle \varphi, V \varphi \rangle_{L^2_y} \psi & 0 \\ 0  & \langle \psi, h \psi \rangle_{L^2_x}  \varphi \end{matrix} \right) .
\]
Clearly, $\mathbb H$ is the generator of a strongly continuous unitary Schr\"odinger group $U(t):= e^{-i t \mathbb H}$, which can be used to rewrite the system using Duhamel's formula as 
\begin{equation}\label{duhamel}
\Psi(t,\cdot) = U(t) \Psi_{\rm in}(\cdot)  - i \int_0^t U(t-s) f(\Psi(\cdot, s)) \, ds.
\end{equation}
Following classical semi-group arguments, cf. \cite{Caz}, it suffices to show that $f(\Psi)$ is locally Lipschitz in $H^1(\R^{d+n})$ 
in order to infer the existence of a unique local in-time solution $\Psi \in C([0,T), H^1(\R^{d+n}))$. This is not hard to show, since, for example:
\begin{align*}
\| f(\Psi_1) - f(\Psi_2) \|_{L^2} \le \| V \|_{L^\infty} (\| \Psi_1 \|_{L^2}^2  + \| \Psi_2 \|_{H^1}^2) \|\Psi_1 - \Psi_2 \|_{L^2},
\end{align*}
in view of the fact that $V\in \mathcal S$. A similar argument can be done for $\| \nabla f(\Psi) \|_{L^2}$ and hence, one gets that there exists a $C=C(V, \|\Psi_1 \|_{H^1},  \|\Psi _2\|_{H^1})>0$ such that
\begin{align*}
\| f(\Psi_1) - f(\Psi_2) \|_{H^1} \le  C   \|\Psi_1 - \Psi_2 \|_{H^1}.
\end{align*}
Using this, \cite[Theorem 3.3.9]{Caz} implies the existence of a $T=T(\| \Psi \|_{H^1})>0$ and a unique solution $\Psi \in C([0,T), H^1(\R^{d+n}))$ of \eqref{duhamel}. It is then also clear, that 
this solution satisfies the conservation of mass and energy for all $t\in [0,T)$. 
Moreover, the quoted theorem also implies that if $T< +\infty$, then
\begin{equation}\label{blowup}
\lim_{t \to T_-} \| \Psi (\cdot, t) \|_{H^1} = \infty.
\end{equation}
However, having  in mind the conservation laws for mass and energy stated in Lemmas \ref{lem:mass} and \ref{lem:energy} together with the fact that we assume w.l.o.g. $V(x,y)\ge 0$, 
we immediately infer that $ \| \Psi (\cdot, t) \|_{H^1}  \le C$ for all $t\in \R$ and hence, the blow-up alternative \eqref{blowup} implies global in-time existence of the obtained solution.
\end{proof}

\begin{remark}
Note that this existence result rests on the fact that the term $\Lambda^{\eps, \delta}(y,t):= \langle \psi^{\eps, \delta},h^\delta \psi^{\eps, \delta} \rangle_{L_x^{2}}$ is interpreted in 
a weak sense, see \eqref{eq:Lambda}. In order to interpret it in a strong sense, one would need to require higher regularity, in particular $\psi^{\eps, \delta} \in H^2(\R^d_x)$.
\end{remark}


\section{Review of Wigner transforms and Wigner measures}\label{sec:Wigner}

The use of Wigner transformation and Wigner measures in the analysis of (semi)-classical asymptotic is, by now, very well established. 
We shall in the following, briefly recall the main results developed in \cite{LiPa, GMMP} (see also \cite{GaMa, MaMa, SMM} for further applications and discussions of Wigner measures):

Denote by $\{ f^\eps \}_{0<\eps \le 1}$ a family of functions $f^\eps\in L^2(\R^d)$, depending continuously on a small parameter $\eps>0$, and by
$$
(\mathcal F_{x\to \xi} f^\e)(\xi)\equiv \widehat f^\e(\xi):= \int_{\R^d} f^\e(x) e^{- i x\cdot \xi} dx.
$$
the corresponding Fourier transform. The associated $\eps$-scaled Wigner transform is then given by \cite{Wi}:
\begin{equation}\label{wig}
w^\e [f^\e](x,\xi): = \frac{1}{(2\pi )^d} \int_{\R^d}
f^\e\left(x-\frac{\e}{2}z 
\right)\overline{f^\e} \left(x+\frac{\e}{2}z
\right)e^{i z \cdot \xi}\,  dz. 
\end{equation}
Clearly, one has
\[
(\mathcal F_{\xi \to z}w)(x,z)= \int_{\R^d} w(x,\xi) e^{- i z \cdot \xi} d\xi =f^\e \Big (x+
\frac{\e}{2} z\Big) \overline{f^\e} \Big(x- \frac{\e}{2} z\Big),
\]
and thus Plancherel's theorem together with a simple change of variables yields
$$
\|w^\e  \|_{L^2(\R^{2d}) }= \e^{-d} (2 \pi)^{-d/2} \|f^\e  \|^2_{L^2(\R^{d})} .
$$
The real-valued function $w^\e(x,\xi) $ acts as a quantum mechanical analogue for classical phase-space distributions. However, $w^\eps(x,\xi)\not \ge0$ in general. 
A straightforward computation shows that the position density associated to $f^\eps$ can be computed via
\begin{equation*}\label{Wdensity}
|f^\eps(x)|^2= \int_{\R^d} w^\e (x,\xi) \, d\xi .
\end{equation*}
Moreover, by taking higher order moments in $\xi$ one (formally) finds
\[
\eps \text{Im} (\overline f^\eps(x) \nabla f^\eps(x)) = \int_{\R^d} \xi w^\e (x,\xi) \, d\xi .
\]
and 
\[
\eps^2 |\nabla f^\eps(x)|^2 = \int_{\R^d} |\xi|^2 w^\e (x,\xi) \, d\xi .
\]
In order to make these computations rigorous, the integrals on the r.h.s. have to be understood in an appropriate sense, since $w^\e \not \in L^1(\R^m_x\times \R^m_\xi)$ in general, 
cf. \cite{LiPa} for more details. 

It has been proved in \cite{GMMP, LiPa} that if $f^\eps$ is uniformly bounded in $L^2(\R^d)$ as $\eps \to 0_+$, i.e., if 
\[\sup_{0< \eps \le 1} \| f^\eps \|_{L_x^2} \le C,\] where $C>0$ is an $\eps$-independent constant, then 
the set of Wigner functions $\{w^\e\}_{0<\e\le 1} \subset \mathcal S'(\R^d_x \times \R^d_\xi)$ is weak$^\ast$ compact.
Thus, up to extraction of sub-sequences, there exists a limiting object $w^0\equiv w\in \mathcal S'(\R^d_x \times \R^d_\xi)$ such that
$$
w^\e[f^\e ]  \stackrel{\e\rightarrow 0_+
}{\longrightarrow}  \mu \quad \text{in $\mathcal S'(\R^d_x \times \R^d_\xi) \, {\rm weak}^\ast$} .
$$
It turns out that the limit is in fact a non-negative Radon measure on phase-space $\mu \in \mathcal M^+(\R^d_x \times \R^d_p)$, called the Wigner measure 
(or, semi-classical defect measure) of $f^\e$, cf. \cite[Theorem III.1]{LiPa}. If, in addition it also holds that $f^\e$ is {\it $\e$-oscillatory}, i.e., 
\[\sup_{0< \eps \le 1} \| \eps \nabla f^\eps \|_{L_x^2} \le C,\]
then one also gets (up to extraction of sub-sequences)
\[
|f^\eps(x)|^2 \stackrel{\e\rightarrow 0_+}{\longrightarrow}  \int_{\R^d} \mu(x, d\xi) ,
\]
in $\mathcal M^+(\R_x^m) \, {\rm weak}^\ast$. Indeed, the Wigner measure $\mu$ is known to encode the classical limit of all physical observables. More precisely, 
for the expectation value of any Weyl-quantized operator $\text{Op}^\e(a)$, corresponding to a classical symbol 
$a(x, p) \in \mathcal S(\R^d_x\times \R^d_\xi)$, one finds as in \cite{GMMP}
\[
\langle f^\e, \text{Op}^\e(a)f^\e\rangle_{L_x^2} =  \iint_{\R^{2d}} a(x,p) w^\eps[f^\eps](dx, d\xi) , 
\]
and hence
\begin{equation}\label{eq:expectation}
\lim_{\eps \to 0_+}\langle f^\e, \text{Op}^\e(a)f^\e\rangle_{L_x^2}  =   \iint_{\R^{2d}} a(x,p) \mu(dx, d\xi) ,
\end{equation}
where the right hand side resembles the usual formula from classical statistical mechanics. 

\begin{remark}
A different topology to study the classical limit of $w^{\eps}$, as $\eps \to 0_+$, was introduced in \cite[Proposition III.1]{LiPa}.
It requires less regularity on the test functions which, in particular, allows to consider much rougher potentials $V$ not necessarily in $\mathcal S(\R^d_x\times \R^n_y)$.
\end{remark}

Finally, we recall that if $f^\e\in C_{\rm b}(\R_t;L^2(\R^d))$ solves a semi-classically scaled Schr\"odinger equation of the form 
\[
i \eps \partial_t f^\e = -\frac{\eps^2}{2}\Delta f^\e + V(x) f^\e, \quad f^\e_{\mid t=0}=f_{\rm in}^\e(x),
\]
then the associated Wigner transformed equation for $w^\e\equiv w^\e[f^\e]$ reads
\[
\partial_t w^\e + \xi\cdot \nabla_x w^\e + \Theta^\e[V]w^\e =0, \quad w^\e_{\mid t=0}=w_{\rm in}^\e(x,\xi),
\]
where $w_{\rm in}^\e\equiv w^\e[\psi_{\rm in}^\e]$ and $\Theta^\varepsilon  [V]$ is a pseudo-differential operator given by \cite{LiPa}:
\begin{align*}
(\Theta^\e [V] w^\e)(x,\xi, t) := -\frac{i}{(2\pi )^d } \iint_{\mathbb R^{2d}} \delta V^\e (x,y)   w^\e (x,\zeta, t)\ e^{i y \cdot (\xi- \zeta)} \, dy \, d\zeta .
\end{align*}
Here, the symbol $\delta V^\e$ is found to be
\begin{align*}\label{delta}
\delta V^\e (x,y)= \frac{1}{\e} \left( V\Big (x+
\frac{\e}{2} y\Big)-V\Big(x- \frac{\e}{2} y\Big)\right).
\end{align*}
Under sufficient regularity assumptions on $V$, one consequently obtains \[
\delta V^\e \stackrel{\e\rightarrow 0_+
}{\longrightarrow} y \cdot \nabla_x V(x).\] 
Using this, it can be proved, see \cite{LiPa, GMMP}, that the Wigner measures  $\mu$ solves \emph{Liouville's equation} on phase space, i.e.
\begin{equation}\label{eq:liouville}
\partial_t \mu + \diver_x( \xi  \mu)  - \diver_\xi (\nabla_x V(x) \mu ) = 0, \quad \mu_{\mid t=0}=\mu_{\rm in}(x,\xi),
\end{equation}
in the sense of distributions $\mathcal D'( \R^d_x\times\R^d_\xi\times \R_t)$. Here, $\mu_{\rm in}$ is the weak$^\ast$ limit of $w^\e_{\rm in}$, along sub-sequences. 
With some further effort (and sufficient regularity assumptions on $V$) one can then show that indeed, $\mu \in C_{\rm b}(\R_t; \mathcal M^+(\R^d_x\times \R^d_\xi))$ satisfying, 
for any test-function $\chi\in \mathcal S(\R^n_x \times \R^n_\xi)$,
\[
\iint_{\R^{2n} } \chi(x,\xi) \mu(x,\xi, t) \, dx\, d\xi = \iint_{\R^{2n}}  \chi(\Phi_{t}(x,\xi)) \mu_{\rm in}(dx,d\xi),
\]
where $\Phi_t:\R^{2d}\to \R^{2d}$ is the Hamiltonian flow associated to \eqref{eq:liouville}:
\begin{equation}
\left\{ \begin{split}
&\dot x(t) = \xi(t), \quad x(0)= x_0\in \R^d,\\
&\dot \xi(t) = -\nabla_x V (x(t)), \quad \xi(0)=\xi_0\in \R^d.
\end{split}\right.
\end{equation}
This allows to prove uniqueness of the weak solution of \eqref{eq:liouville}, provided 
the initial measure $\mu_{\rm in}$ is the same for all sub-sequences $\{\eps_n\}_{n\in \N}$ of $w^\eps[\varphi^\eps_{\rm in }]$, see \cite[Theorem IV.1]{LiPa}.


\section{The mixed quantum-classical limit}\label{sec:semi} 
In this section we will investigate the semi-classical limit of the TDSCF system \eqref{eq:TDSCF}, which corresponds to the case $\eps \to 0_+$ and $\delta = O(1)$ {\it fixed}. 
In other words, we want to pass to the classical limit in the equation for $\varphi^{\eps, \delta}$ only, while retaining the full quantum mechanical dynamics for $\psi^{\eps, \delta}$. 
To this end, we introduce the $\eps$-scaled Wigner transformation of $\varphi^{\eps, \delta}$ in the form
\[
w^\e [\varphi^{\e, \delta}](y,\eta, t): = \frac{1}{(2\pi)^n} \int_{\R^n}
\varphi^{\e, \delta}\left(y-\frac{\e}{2}z , t\right)\overline{\varphi^{\e, \delta}} \left(y+\frac{\e}{2}z, t \right)e^{i z \cdot \eta}\,  dz. 
\]
In this subsection, we could, in principle, suppress the dependence on $\delta$ completely (since it is assumed to be fixed), 
but since we will consider the subsequent $\delta \to 0_+$ limit in Section \ref{sec:classical}, 
we shall keep its appearance within the superscript.

Recalling that initially $\| \varphi^{\eps}_{\rm in} \|_{L^2_y}=1$, the a-priori estimates established in 
Lemma \ref{lem:mass} and Lemma \ref{lem:energy} together with the fact that $V(x,y)\ge0$,
imply the uniform (in $\eps$) bounds
\[
\sup_{0< \eps \le 1} (\| \varphi^{\e, \delta} (t,\cdot) \|_{L^2_y} + \| \eps \nabla \varphi^{\e, \delta} (t,\cdot) \|_{L^2_y}) \le C(t)
\]
for any $t\in \R$, where $C(t)\ge0$ is a constant independent of $\eps$ and $\delta$. In other words, $\varphi^{\e, \delta}$ is $\eps$-oscillatory and we consequently infer the existence of a limiting 
Wigner measure $\mu^{0,\delta}\equiv \mu^\delta\in \mathcal M^+(\R^d_y \times \R^d_\eta)$ such that (up to extraction of sub-sequences)
\[
w^\eps [\varphi^{\e, \delta}] \stackrel{\e\rightarrow 0_+
}{\longrightarrow}  \mu^\delta \quad
\text{in $L^\infty(\R_t; \mathcal S'(\R^n_y \times \R^n_\eta)) \, 
{\rm weak}^\ast$},
\]
together with 
\[
|\varphi^{\e, \delta} (y,t)|^2 \stackrel{\e\rightarrow 0_+}{\longrightarrow}  \int_{\R^n_\eta} \mu^\delta(y, d\eta, t). 
\]
The measure $\mu^\delta$ encodes the classical limit of the subsystem described by the $y$-variables only. In addition, having in mind 
our assumption that $V\in \mathcal S(\R^d_x\times \R^n_y)$, we directly infer that 
\begin{align*}
\Upsilon^{\eps, \delta}(x,t) := &\ \int_{\R_y^n} V(x,y) |\varphi^{\eps, \delta}(y,t)|^2 \, dy \\
&\ \stackrel{\e\rightarrow 0_+
}{\longrightarrow}  \iint_{\R_{y,\eta}^{2n}} V(x,y) \mu^\delta( dy, d\eta, t)\equiv \Upsilon^{\delta}(x,t),
\end{align*}
point-wise for all $(x,t)\in \R^{d+1}$. The right hand side of the above relation describes  
the classical limit of the self-consistent potential $\Upsilon^{\eps, \delta}$ obtained through the Wigner measure of $\varphi^{\eps, \delta}$. 
Note that both $\Upsilon^{\eps, \delta}, \Upsilon^{ \delta} \in C(\R_t;\mathcal S(\R_x^d))$, in view of our assumption \eqref{eq:hyp} on $V$ and the existence result for 
the solution $\varphi^{\eps, \delta}$, cf. Propositon \ref{prop:existence}. Moreover, $\Upsilon^{\eps, \delta}$ is uniformly bounded in $\e$, since
\begin{equation}\label{eq:Ubound}
| \Upsilon^{\eps, \delta}(y,t) | \le \sup_{x,y} |V(x,y)| \, \| \varphi^{\eps, \delta} (\cdot, t) \|^2_{L_x^2} \le \| V \|_{L^\infty},
\end{equation}
since $\| \varphi^{\eps, \delta} (\cdot, t) \|_{L_x^2} =1$, $\forall \, t\in \R$ in view of the a-priori estimate of Lemma \ref{lem:mass}. 

The following Proposition shows that 
the solution of the first equation within the TDSCF system \eqref{eq:TDSCF} stays close to the one where the potential
$\Upsilon^{\eps, \delta}$ is replaced by its limit $\Upsilon^{\delta}$.

\begin{proposition}\label{prop:potapprox} 
Let $V\in\mathcal S(\R^d_x\times \R^n_y)$ and $\psi^{\eps, \delta}$, $\psi^{\delta}\in C(\R_t; H^1(\R^d_x))$ solve, respectively
\[
i \delta \partial_t \psi^{\eps, \delta}  = \left (-\frac{\delta^2}{2}\Delta_x  + \Upsilon^{\eps, \delta}(x,t)  \right) \psi^{\eps, \delta} \, ,
\quad \psi^{\eps, \delta} _{\mid t=0} = \psi^\delta_{\rm in}(x),
\]
and
\[
i \delta \partial_t \psi^{\delta}  = \left (-\frac{\delta^2}{2}\Delta_x  + \Upsilon^{\delta}(x,t)  \right) \psi^{\delta} \, ,
\quad \psi^{\delta} _{\mid t=0} = \psi^\delta_{\rm in}(x),
\]
then, for any $T>0$ 
\[
\sup_{t\in [0,T]} \| \psi^{\eps, \delta}(\cdot, t) - \psi^{\delta}(\cdot, t) \|_{L^2_x} \stackrel{\eps \rightarrow 0}{\longrightarrow} 0.
\]
\end{proposition}

\begin{proof} Denote the Hamiltonian operators corresponding to the above equations by
\[
H^{\eps, \delta}_1=-\frac{\delta^2}{2}\Delta_x +\Upsilon^{\eps, \delta}(x,t), \quad 
H^{\delta}_2=-\frac{\delta^2}{2}\Delta_x + \Upsilon^{\delta}(x,t).
\]
In view of our assumptions on the potential $V$ 
and the existence result given in Proposition \ref{prop:existence}, we infer that $H_1$ and $H_2$ 
are essentially self-adjoint on $L^2(\R_x^d)$ and hence they generate unitary propagators $\mathcal{U}^{\eps, \delta}_1(t,s)$ and $\mathcal{U}^{\delta}_2(t,s)$, such that
\[
\mathcal{U}^{\eps, \delta}_1(t,s)\psi^{\varepsilon, \delta}(x,s)=\psi^{\varepsilon, \delta}(x,t),\quad \mathcal{U}^{ \delta}_2(t,s)\psi^{\delta}(x,s)=\psi^{\delta}(x,t).
\]
Therefore, one obtains
\begin{align*}
 \| \psi^{\eps, \delta}(\cdot, t) - \psi^{\delta}(\cdot, t) \|_{L^2_x} = & \ \big \| \, \mathcal{U}^{\eps, \delta}_1(t,s)\psi^{\varepsilon, \delta}_{\rm in} (\cdot) - \psi^{\delta}(\cdot, t) \, \big \|_{L^2_x} \\
  = & \ \big \| \, \psi^{\varepsilon, \delta}_{\rm in} (\cdot) - \mathcal{U}^{\eps, \delta}_1(0,t) \psi^{\delta}(\cdot, t) \, \big \|_{L^2_x}\\
   = & \ \left\Vert \int_0^t \frac{d}{ds} \left(\mathcal{U}^{\eps, \delta}_1(0,s) \psi^{\delta}(\cdot, s) \right) ds \right\Vert_{L^2_x} ,
 \end{align*} 
using $(\mathcal{U}^{\eps, \delta}_1)^{-1}(t,s)=\mathcal{U}^{\eps, \delta}_1(s,t)$. Computing further, one gets
\begin{align*}  
 \| \psi^{\eps, \delta}(\cdot, t) - \psi^{\delta}(\cdot, t) \|_{L^2_x}    = &\, \left\Vert \int_0^t \Big(\frac{d}{ds} \mathcal{U}^{\eps, \delta}_1(0,s)\Big)\psi^{\delta}(\cdot, s)
 +\mathcal{U}^{\eps, \delta}_1(0,s)\frac{d}{ds} \psi^{\delta}(\cdot, s) ds \right\Vert_{L^2_x}  \\
   = & \, \left\Vert  \int_0^t  \mathcal{U}^{\eps, \delta}_1(0,s) \left( H^{\eps, \delta}_1\psi^{\delta}(\cdot, s) - H^{\delta}_2 \psi^{\delta}(\cdot, s) \right) ds \right\Vert _{L^2_x}  \\
   = & \, \left\Vert  \int_0^t  \mathcal{U}^{\eps, \delta}_1(0,s) \left( \Upsilon^{\eps, \delta}(\cdot, s) - \Upsilon^{\delta}(\cdot, s) \right) \psi^{\delta}(\cdot, s) \, ds \right\Vert _{L^2_x}  .
\end{align*}
By Minkowski's inequality, one thus has
\[
 \| \psi^{\eps, \delta}(\cdot, t) - \psi^{\delta}(\cdot, t) \|_{L^2_x}  \le\int_0^t   \left\Vert    \left( \Upsilon^{\eps, \delta}(\cdot, s) - \Upsilon^{\delta}(\cdot, s) \right) \psi^{\delta}(\cdot, s)  \right\Vert _{L^2_x} \, ds,
\]
and hence,
\[
\sup_{t\in [0,T]} \| \psi^{\eps, \delta}(\cdot, t) - \psi^{\delta}(\cdot, t) \|_{L^2_x} \le C_T \sup_{t\in [0,T]} \big \| \left( \Upsilon^{\eps, \delta}(\cdot, t) - \Upsilon^{\delta}(\cdot, t) \right) \psi^{\delta}(\cdot, t) \big \|_{L^2_x}. 
\]
Now, since  $\Upsilon^{\eps, \delta}(\cdot, t)$ is bounded in $L^\infty$ uniformly in $\e$, cf. \eqref{eq:Ubound}, and since
\[
\left( \Upsilon^{\eps, \delta}(x, t) - \Upsilon^{\delta}(x, t)\right) \stackrel{\e\rightarrow 0_+
}{\longrightarrow} 0,
\]
point-wise in $x$, Lebesgue's dominated convergence theorem is sufficient to conclude the desired result.
\end{proof}

In order to identify the limiting measure $\mu^\delta$ we shall derive the corresponding evolutionary system. As a first step, we recall 
$w^{\e, \delta}\equiv w^{\eps}[\varphi^{\eps, \delta}]$ solves
\begin{equation} \label{wignereq}
\partial _t w^{\e, \delta}+\eta \cdot \nabla_y w^{\e, \delta} + \Theta [\Lambda^{\eps, \delta}]w^{\varepsilon, \delta} =0, 
\end{equation}
where $\Theta [\Lambda^{\eps, \delta}]$ is explicitly given by
\begin{align*}
\Theta [\Lambda^{\eps, \delta}]w^{\varepsilon, \delta}(y, \eta, t) := -\frac{i}{(2\pi )^n } \iint_{\mathbb R^{2n}} \delta \Lambda^{\eps, \delta}(y,z, t)  w^{\varepsilon, \delta}
 (y,\zeta, t)\ e^{i z \cdot (\eta- \zeta)} \, dz \, d\zeta ,
\end{align*}
and the associated symbol $\delta \Lambda^{\eps, \delta}$ reads
\begin{align*}
\delta \Lambda^{\eps, \delta} (y, z, t) & = \frac{1}{\e} \left( \Lambda^{\eps, \delta}\big (y+
\frac{\e}{2} z, t\big)-\Lambda^{\eps, \delta}\big(y- \frac{\e}{2} z, t\big)\right) \\
& = \frac{1}{\e} \left(
\langle \psi^{\eps, \delta},V \psi^{\eps, \delta} \rangle_{L_x^{2}} \big (y+
\frac{\e}{2} z, t\big)
-\langle \psi^{\eps, \delta},V \psi^{\eps, \delta} \rangle_{L_x^{2}}\big (y-
\frac{\e}{2} z, t\big) \right),
\end{align*}
in view of \eqref{eq:Lambda}. In particular, this shows that the purely time-dependent term $\vartheta^{\eps, \delta}(t)$ does contribute to the symbol of the pseudo-differential operator. 
The same would have been true if 
we would have used the time-dependent Gauge transformation \eqref{eq:gauge} from the beginning. Introducing the short hand notation 
\[
\mathcal V^{\eps, \delta}(y,t):=\langle \psi^{\eps, \delta}(\cdot, t),V(\cdot, y) \psi^{\eps, \delta}(\cdot, t) \rangle_{L_x^{2}},\] 
one can rewrite
\begin{align*}
\delta \Lambda^{\eps, \delta} (y, z, t) = & \,
 \frac{1}{\e} \left(
\mathcal V^{\eps, \delta} \big (y+
\frac{\e}{2} z, t\big)
-\mathcal V^{\eps, \delta}\big (y-
\frac{\e}{2} z, t\big) \right),
\end{align*}
and thus $\Theta [\Lambda^{\eps, \delta}]\equiv \Theta [\mathcal V^{\eps, \delta}]$. Note that $\mathcal V^{\eps, \delta}\in C(\R_t;\mathcal S(\R_y^n))$.
However, the main difference to the case of a given potential $V$ is that here $\mathcal V^{\eps, \delta}$ itself depends on $\eps$ and is computed self-consistently 
from the solution of $\psi^{\e, \delta}$. We can therefore not directly apply the Wigner transformation results of \cite{GMMP}. We nevertheless shall prove in the 
following proposition that the limit of $\Theta [\mathcal V^{\eps, \delta}]$ as $\eps\to 0_+$ is indeed what one would formally expect it to be.

\begin{proposition}\label{prop:wignerpotapprox}
Let $V\in \mathcal S(\R_x^{d}\times \R^d_y)$ and $\psi^{\eps, \delta}$, $\psi^{\delta}\in C(\R_t; H^1(\R^d_x))$, then, up to selection of another sub-sequence
\[
\Theta [\Lambda^{\eps, \delta}]w^{\varepsilon, \delta}\stackrel{\e\rightarrow 0_+
}{\longrightarrow}   F^{\delta}(y,t) \cdot \nabla_\eta \mu^\delta \quad
\text{in $L^\infty([0,T]; \mathcal S'(\R^n_y \times \R^n_\eta)) \, 
{\rm weak}^\ast$},
\]
where the semi-classcial force is defined by 
\[
F^{\delta}(y,t):= - \int_{\R^d} \nabla_y V(x,y) |\psi^{\delta}(x,t)|^2\, dx.
\]
\end{proposition}

\begin{proof} 
Denote $\mathcal V^{\delta}(y,t)=\langle \psi^{\delta}(\cdot, t),V(\cdot, y) \psi^{ \delta}(\cdot, t) \rangle_{L_x^{2}}$. Then, we can estimate 
\begin{align*}
|\mathcal V^{\eps, \delta}(y,t) - \mathcal V^{0, \delta}(y,t)|\le & \ \| V \|_{L^\infty} \int_{\R^d} \big| |\psi^{\eps,\delta}(x,t)|^2 - | \psi^{\delta}(x,t)|^2 \big| dx \\
\le & \ 2  \ \| V \|_{L^\infty} \| \psi^{\eps,\delta}(\cdot,t) -  \psi^{\delta}(\cdot, t)\|_{L_x^2},
\end{align*}
where in the second inequality we have used the Cauchy-Schwarz inequality together with the fact that $||a|^2-|b|^2| \le |a-b| (|a|+|b|)$ for any $a,b\in \C$. 
The strong $L^2$-convergence of $\psi^{\eps,\delta}$ stated in Proposition \ref{prop:potapprox} therefore implies 
\[
\mathcal V^{\eps, \delta}(y,t) \stackrel{\e\rightarrow 0_+
}{\longrightarrow}  \mathcal V^{0, \delta}(y,t) \equiv \langle \psi^{\delta},V \psi^{\delta} \rangle_{L_x^{2}}(y,t),
\]
point-wise. Analogously we infer point-wise convergence of $\nabla_y \mathcal V^{\eps, \delta} \to \nabla_y \mathcal V^{\delta}$.

Next, we note that $\mathcal V^{\eps, \delta}$ is uniformly bounded in $\e$, since, as before,
\[
| \mathcal V^{\eps, \delta}(y,t) | \le \sup_{x,y} |V(x,y)| \, \| \psi^{\eps, \delta} (\cdot, t) \|^2_{L_x^2} \le \| V \|_{L^\infty},
\] 
having in mind that $\| \psi^{\eps, \delta} (\cdot, t) \|_{L_x^2} =1$, $\forall \, t\in \R$. Since $V\in \mathcal S(\R^d_x\times \R^n_y)$ the same argument also applies to $\nabla_y \mathcal V^{\eps, \delta}$. 
Moreover, by using the Mean-Value Theorem, we can estimate
\[
| \nabla_y \mathcal V^{\eps, \delta}(y_1,t) - \nabla_y \mathcal V^{\eps, \delta}(y_2,t) |  \le |y_1 - y_2| \sup_{x,y} |D^2 V| \|\psi^{\eps, \delta}(\cdot, t) \|_{L^2_x} \le C |y_1 - y_2| .
\]
This shows that $F^{\eps, \delta}:=-\nabla_y \mathcal V^{\eps, \delta}$ is equicontinuous in $y$, 
and hence the Arzela-Ascoli Theorem guarantees that there exists a subsequence, such that $F^{\eps, \delta}$ converges, as $\eps \to 0_+$, 
uniformly on compact sets in $y,t$. 

Now, let $\chi\in \mathcal S(\R^n_y \times \R^n_\eta)$ be a test-function with the property that its Fourier transform with respect to $\eta$, i.e.
\[
(\mathcal F_{\eta\to z} \chi)(y,z) \equiv\widetilde \chi (y,z)= \int_{\R^n} \chi(y,\eta) e^{-i \eta\cdot z} \, d\eta,
\]
has compact support with respect to both $y$ and $z$. This kind of test functions are dense in $\mathcal S(\R^n_y \times \R^n_\eta)$ and hence it suffices to show the 
assertion for the $\chi$ only. Multiplying by $\chi$ and integrating allows one to write
\[
 \langle \Theta [\Lambda^{\eps, \delta}]w^{\varepsilon, \delta}, \chi \rangle  =   - \langle \, w^{\varepsilon, \delta}, \Xi^{\eps, \delta} \rangle_{\mathcal S', \mathcal S},
\] 
where
\[
\Xi^{\eps, \delta} (y,\eta)=\frac{i}{(2\pi )^n } \int_{\R^n} \widetilde \chi(y,z) e^{i z\cdot \eta} \frac{1}{\e} \left(
\mathcal V^{\eps, \delta}\big (y+
\frac{\e}{2} z, t\big)
-\mathcal V^{\eps, \delta}\big (y-
\frac{\e}{2} z, t\big) \right) \, dz \in \mathcal S.
\]
Since, $\widetilde \chi$ has compact support the uniform convergence of $F^{\e, \delta}$ allows us to conclude
\[
\Xi^{\eps, \delta} \stackrel{\e\rightarrow 0_+ }{\longrightarrow}  i \nabla_y \mathcal V^{\delta} (y,t) \cdot \mathcal F_{z\to \eta}^{-1}( z \widetilde \chi(y,z))(y, \eta) \equiv 
F^\delta(y,t) \cdot \nabla_\eta \chi(y,\eta) .
\]
\end{proof}

\begin{remark}
One should note that, even though $\Lambda^{\eps, \delta}$ is a self-consistent potential, depending nonlinearly upon the solution $\psi^{\eps, \delta}$, the 
convergence proof given above is very similar to the linear case \cite{LiPa}, due to the particular structure of the nonlinearity. In particular, we do not require 
to pass to the mixed state formulation which is needed to establish the classical limit in other self-consistent quantum dynamical models, as for example in \cite{MaMa}. 
\end{remark}

In summary, this leads to the first main result of our work, which shows that the solution to \eqref{eq:TDSCF}, as $\eps \to 0_+$ (and with $\delta =O(1)$ fixed) 
converges to a mixed quantum-classical system, consisting of a Schr\"odinger equation for the $x$-variables and a classical Liouville equation for the $y$-variables.

\begin{theorem} \label{thm:mixed}
Let $V\in \mathcal S(\R^d_x\times \R^n_y)$ and $\psi^{\eps, \delta}$, $\varphi^{\eps, \delta}\in C(\R_t; H^1(\R^d_x))$, be solutions to the TDSCF system 
\eqref{eq:TDSCF} with uniformly bounded initial mass and energy, i.e, $M^{\eps, \delta}(0)\le C_1$, $E^{\eps, \delta}(0)\le C_2$. Then, for any $T>0$, it holds
\[
\psi^{\eps, \delta} \stackrel{\e\rightarrow 0_+
}{\longrightarrow} \psi^\delta \quad
\text{in $L^\infty([0,T]; L_y^2(\R^n)) $},
\]
and
\[
w^\eps [\varphi^{\e, \delta}] \stackrel{\e\rightarrow 0_+
}{\longrightarrow}  \mu^\delta \quad
\text{in $L^\infty([0,T]; \mathcal S'(\R^n_y \times \R^n_\eta)) \, 
{\rm weak}^\ast$},
\]
where $\psi^\delta\in C(\R_t; L^2(\R^d))$ and $\mu^\delta\in C(\R_t; \mathcal M^+(\R_y^n\times \R^n_\eta))$ solve the following mixed quantum-classical system
\begin{equation}\label{eq:mixed}
\left\{ \begin{split}
& i \delta \partial_t \psi^{\delta}  = \left (-\frac{\delta^2}{2}\Delta_x  + \Upsilon^{\delta} (x,t) \right) \psi^{\delta} \, ,
\quad \psi^{\delta} _{\mid t=0} = \psi^\delta_{\rm in}(x),\\
& \partial_t \mu^{\delta} + \diver_y( \eta \mu^{\delta}) + \diver_\eta (F^{\delta}(y,t) \mu^{\delta}) = 0 \, ,
\quad \mu^\delta_{\mid t=0} = \mu_{\rm in}(y, \eta).
\end{split}\right.
\end{equation}
Here $\mu_{\rm in}$ is the initial Wigner measure obtained as the weak$^\ast$ limit of $w^\eps[\varphi^\eps_{\rm in }]$ and
\[
\Upsilon^{\delta}(x,t)  = \iint_{\R^{2n}} V(x,y) \mu^\delta( dy, d\eta, t), \quad
F^{\delta}(y,t)=- \int_{\R^d} \nabla_y V(x,y) |\psi^{\delta}(x,t)|^2\, dx.
\]
\end{theorem}

\begin{proof}
In view of Proposition \ref{prop:potapprox} and Proposition \ref{prop:wignerpotapprox}, the result follows directly from the Wigner measure techniques 
established in \cite{GMMP, LiPa}. In particular the continuity in-time of the Wigner measure $\mu^\delta$ can be inferred by following the arguments of 
\cite[Theorem IV.1]{LiPa}.
\end{proof}

\begin{remark}
It is possible to obtain slightly stronger convergence results with respect to time, by first proving that 
$\partial_t w^{\varepsilon, \delta}$ is bounded in $L^\infty((0,T),\mathcal S'(\R^n_y\times \R^n_\eta))$, which consequently implies time-equicontinuity of $w^{\e, \delta}$, see, e.g., \cite{AMS}.
\end{remark}

\subsection{Connection to the Ehrenfest method}\label{sec:ehrenfest} 
The fact that $F^\delta\in C(\R_t;\mathcal S(\R^d))$ allows us to introduce a smooth globally defined Hamiltonian flow 
$\Phi^\delta_t:\R^{2n}\to \R^{2n}$ induced by:
\begin{equation*}
\left\{ \begin{split}
&\dot y(t) = \eta(t), \quad y(0)= y_0\in \R^n,\\
&\dot \eta(t) =  F^{\delta}(y(t),t), \quad \eta(0)=\eta_0\in \R^n.
\end{split}\right.
\end{equation*}
In particular we have, for any test-function $\chi\in \mathcal S(\R^n_y \times \R^n_\eta)$, the push-forward formula
\[
\iint_{\R^{2n} } \chi(y,\eta) \mu^\delta(y,\eta, t) \, dy\, d\eta = \iint_{\R^{2n}}  \chi(\Phi^\delta_{t}(y,\eta)) \mu_{\rm in}(dy,d\eta).
\]
In particular, if initially $\mu_0(y, \eta)=\delta(y-y_0, \eta - \eta_0)$, i.e. a delta distribution centered at $(y_0, \eta_0)\in \R^{2n}$, this yields
$\mu^\delta(y,\eta, t)= \delta(y-y(t), \eta-\eta(t))$, for all times $t\in \R$. 
Such kind of Wigner measures can be obtained as the classical limit of a particular type of wave functions, called
semi-classical wave packets, or coherent states, see \cite{LiPa}. In this case, we also find
\[
\Upsilon^{\delta}(x,t) := \iint_{\R^{2n}} V(x,y) \mu^\delta( dy, d\eta, t) = V(x,y(t)),
\]
and the mixed quantum-classical system becomes 
\begin{equation}\label{eq:ehrenfest}
\left\{ \begin{split}
& i \delta \partial_t \psi^{\delta}  = \left (-\frac{\delta^2}{2}\Delta_x  + V(x,y(t)) \right) \psi^{\delta} \, ,
\quad \psi^{\delta} _{\mid t=0} = \psi^\delta_{\rm in}(x),\\
&\ddot y(t) = - \int_{\R^d} \nabla_y V(x,y(t)) |\psi^{\delta}(x,t)|^2\, dx, \quad y_{\mid t=0} =y_0, \ \dot{y}_{\mid t=0} =\eta_0,
\end{split}\right.
\end{equation}
with $y_0, \eta_0 \in \R^n$.
This is a well-known model in the physics and quantum chemistry literature, usually referred to as {\it Ehrenfest method}. It has been studied in, e.g, \cite{BS, BNS} in the context of 
quantum molecular dynamics.

\begin{remark}
A closely related scaling-limit is obtained in the case where the time-derivatives in both equations of \eqref{eq:TDSCF} are scaled by the same factor $\eps$. At least formally, this leads 
to an Ehrenfest-type model similar to \eqref{eq:ehrenfest}, but with a stationary instead of a time-dependent Schr\"odinger equation, cf. \cite{BNS, Dr}. 
In this case, connections to the Born-Oppenheimer approximation 
of quantum molecular dynamics become apparent, see, e.g., \cite{SpTe}. From the mathematical point of view this scaling regime combines 
the classical limit for the subsystem described by the $y$-variables with a {\it time-adiabatic limit} for the subsystem described in $x$. However, due to the nonlinear coupling within the TDSCF system \eqref{eq:TDSCF} this scaling limit is highly nontrivial and will be the main focus of a future work.
\end{remark}


\section{The fully classical limit}\label{sec:classical}

In order to get a better understanding (in particular for the expected numerical treatment of our model), we will now turn to the question of 
how to obtain a completely classical approximation for the system \eqref{eq:TDSCF}. There are at least two possible ways to do so.
One is to consider the limit $\delta \to 0_+$ in the obtained mixed quantum-classical system \eqref{eq:mixed}, which in itself corresponds to 
the iterated limit $\eps \to 0_+$ and then $\delta \to 0_+$ of \eqref{eq:TDSCF}. Another possibility is to take $\eps = \delta \to 0_+$ in \eqref{eq:TDSCF}, which 
corresponds to a kind of ``diagonal limit" in the $\eps, \delta$ parameter space.

\subsection{The classical limit of the mixed quantum-classical system} In this section we shall perform the limit $\delta\to 0_+$ of the 
obtained mixed quantum-classical system \eqref{eq:mixed}. To this end, we first introduce the $\delta$-scaled Wigner transform of $\psi^\delta$: 
\[
W^\delta [\psi^{ \delta}](x,\xi, t): = \frac{1}{(2\pi)^d} \int_{\R^d}
\psi^{\delta}\left(x-\frac{\e}{2}z , t\right)\overline{\psi^{ \delta}} \left(x+\frac{\e}{2}z, t \right)e^{i z \cdot \xi}\,  dz. 
\]
The results of Lemma \ref{lem:mass} and Lemma \ref{lem:energy} imply that $\psi^\delta$ is a family of $\delta$-oscillatory functions, i.e, 
\begin{equation}\label{eq:energy_psi}
\sup_{0< \delta \le 1} (\| \psi^{ \delta} (t,\cdot) \|_{L^2_y} + \| \delta \nabla \psi^{ \delta} (t,\cdot) \|_{L^2_y}) \le C(t) 
\end{equation}
and thus there exists a limiting measure $\nu\in \mathcal M^+(\R^d_x\times \R^d_\xi)$, such that 
\[
W^\delta[\psi^\delta] \stackrel{\delta\rightarrow 0_+
}{\longrightarrow}  \nu \quad \text{in $L^\infty(\R_t;\mathcal S'(\R^d_x \times \R^d_\xi)) \, {\rm weak}^\ast$.}
\]
Moreover, since $V\in \mathcal S(\R^{n+d})$ we also have
\begin{align*}
F^{\delta}(y,t)= &- \int_{\R^d} \nabla_y V(x,y) |\psi^{\delta}(x,t)|^2\, dx =-\iint_{\R_{x,\xi}^{2n}}\nabla_y V(x,y) W^\delta( dx, d\xi, t)  \\
& \stackrel{\e\rightarrow 0_+}{\longrightarrow}  -\iint_{\R_{x,\xi}^{2n}}\nabla_y V(x,y) \nu( dx, d\xi, t) \equiv F(y,t),
\end{align*}
point-wise, for all $t\in \R$. By using the same arguments as before (see in particular the proof of Proposition \ref{prop:wignerpotapprox}) we infer that $F^\delta$ is uniformly bounded and 
equicontinuous in $y$, and hence, up to extraction of possibly another sub-sequence, $F^{\delta}$ converges, as $\delta \to 0_+$ 
uniformly on compact sets in $y,t$. 

With the results above, we prove in the following proposition the convergence of the Wigner measure $\mu^\delta$ as $\delta\rightarrow 0_+$.

\begin{proposition}\label{prop:deltawigner} Let $\mu^\delta\in C(\R_t; \mathcal M^+(\R^n_y\times\R^n_\eta))$ be a distributional solution of 
\[\partial_t \mu^{\delta} + \diver_y( \eta  \mu^{\delta} )+\diver_\eta (F^{\delta}(y,t) \mu^{\delta}) = 0,\]
and $\mu\in C(\R_t; \mathcal M^+(\R^n_y\times\R^n_\eta))$ be a distributional solution of 
\[\partial_t \mu+ \diver_y(\eta \mu)+\diver_\eta (F(y,t) \mu )= 0,\]
such that initially $\mu^\delta_{\mid t=0} = \mu _{\mid t=0}$, then 
\[
\mu^\delta \stackrel{\delta\rightarrow 0_+
}{\longrightarrow}  \mu \quad \text{in $L^\infty([0,T];\mathcal M^+(\R^n_y \times \R^n_\eta)) \, {\rm weak}^\ast$.}
\]
\end{proposition}
\begin{proof}
We consider the difference $e^\delta:=\mu^\delta - \mu$. Then $e^\delta(y,\eta, t)$ solves (in the sense of distributions) the following inhomogeneous equation:
\[
\partial_t e^{\delta} + \diver_y( \eta e^{\delta}) +\diver_\eta( F^{\delta}(y,t)\ e^{\delta}) =  \diver_\eta((F(y,t)-F^{\delta}(y,t)) \mu),
\]
subject to $e^\delta_{\mid t=0}  = 0$.
For test-functions of the form $\chi(y, \eta)\sigma(t)\in C_0^\infty$, the inhomogeneity on the right hand side is given by
\begin{align*}
& \langle \chi , ( \diver_\eta((F(y,t)-F^{\delta}(y,t)) \mu)\rangle = \\
& = \int_0^T \sigma(t) \iint_{\R^{2n}} \nabla_\eta \chi(y,\eta) \cdot (F^\delta (y,t)-F(y,t)) \mu(dy, d\eta, t)\, dt.
\end{align*}
In view of the arguments above, this term goes to zero as $\delta \to 0_+$. 
But since $e^\delta(0,y,\eta)=0$, a continuity argument based on (the weak formulation of) Duhamel's formula 
then implies that $e^\delta(t,y,\eta)\to 0$ in $\mathcal M^+(\R^d_x \times \R^d_\xi)$ weak$^\ast$, as $\delta \to 0_+$, for all $t\in [0,T]$.
\end{proof}

By Wigner transforming the first equation in the mixed quantum-classical system \eqref{eq:mixed}, we find that $W^\delta[\psi^\delta]\equiv W^{\delta}$ satisfies
\[
\partial_t W^{\delta}  + \xi \cdot \nabla_x W^\delta + \Theta[\Upsilon^\delta]W^\delta = 0 \, ,
\quad W^{\delta}_{\mid t=0} = W^\delta[\psi^\delta_{\rm in}](x,\xi).
\]
To obtain the convergence of the term $\Theta[\Upsilon^\delta]W^\delta$, we note that with the convergence of the Wigner measure $\mu^\delta$, which is obtained in 
Proposition \ref{prop:deltawigner}, one gets
\begin{align*}
 \Upsilon^{\delta}(x,t)  & =  \iint_{\R^{2n}} V(x,y) \mu^\delta( dy, d\eta, t) \\
& \stackrel{\delta\rightarrow 0_+}{\longrightarrow}  \iint_{\R^{2n}} V(x,y) \mu( dy, d\eta, t)  \equiv \Upsilon^{}(x,t) 
\end{align*}
point-wise, for all $t\in \R$. Similar to previous cases, one concludes that, up to extraction of possibly another sub-sequence, $\Upsilon^{\delta}$ converges, as $\delta \to 0_+$, 
uniformly on compact sets in $x,t$. 

With the same techniques as in the proof of Proposition \ref{prop:wignerpotapprox}, one can then derive the equation for the associated Wigner measure $\nu$. 
The classical limit of the mixed quantum-classical system can thus be summarized as follows.

\begin{theorem} \label{thm:classical1}
Let $V\in \mathcal S(\R^{d+n})$, and $\psi^{\delta}\in C(\R_t; H^1(\R^d_x))$, $\mu^\delta\in C(\R_t; \mathcal M^+(\R^n_y\times\R^n_\eta))$ be solutions to the system 
\eqref{eq:mixed} with uniformly bounded initial mass. We also assume $\psi^{\delta}$ has uniformly bounded energy. Then, for any $T>0$, it holds that, the Wigner transform
\[
W^{\delta} \stackrel{\delta\rightarrow 0_+
}{\longrightarrow} \nu \quad
\text{in $L^\infty([0,T]; \mathcal S'(\R^d_x \times \R^d_\xi)) \, 
{\rm weak}^\ast$},
\]
and the Wigner measure
\[
\mu^{\delta} \stackrel{\delta\rightarrow 0_+
}{\longrightarrow}  \mu \quad
 \text{in $L^\infty([0,T];\mathcal M^+(\R^n_y \times \R^n_\eta)) \, {\rm weak}^\ast$.}
\]
where $\nu \in C(\R_t; \mathcal M^+(\R_x^d\times \R^d_\xi))$ and $\mu \in C(\R_t; \mathcal M^+(\R_y^n\times \R^n_\eta))$ solve the following coupled system of Vlasov-type equations
in the sense of distributions
\begin{equation}\label{classicalsystem}
\left\{ \begin{split}
& \partial_t \nu+ \diver_x(\xi  \nu) -\diver_\xi (\nabla_x \Upsilon(x,t) \nu )= 0 \, ,
\quad \nu _{\mid t=0} = \nu_{\rm in}(x,\xi),\\
& \partial_t \mu +\diver_y ( \eta \mu ) + \diver_\eta(F(y,t) \mu )= 0 \, ,
\quad \mu_{\mid t=0} = \mu_{\rm in}(y, \eta).
\end{split}\right.
\end{equation}
Here $\nu_{\rm in}$ is  the initial Wigner measure obtained as the weak$^\ast$ limit of $W^\delta[\psi^\delta_{\rm in }]$, and
\[
\Upsilon(x,t)  = \iint_{\R^{2n}} V(x,y) \mu( dy, d\eta, t), \quad
F(y,t)=-\iint_{\R^{2d}} \nabla_y V(x,y) \nu( dx, d\xi, t).
\]
\end{theorem} 
\begin{remark}
Note that system \eqref{classicalsystem} admits a special solution of the form
\[
\nu(x,\xi, t)= \delta(x-x(t), \xi-\xi(t)),\quad \mu(y,\eta, t)= \delta(y-y(t), \eta-\eta(t)), 
\]
where $x(t), y(t), \xi(t), \eta(t)$ solve the following Hamiltonian system:
\begin{equation*}
\left\{ \begin{split}
&\dot x(t) = \xi(t), \quad \quad  \quad \qquad \quad \ \ x(0)=x_0,\\
& \dot \xi(t)= - \nabla_xV(x(t), y(t)), \quad \xi(0)=\xi_0,\\
&\dot y(t) = \eta(t),  \quad \quad  \quad \qquad \quad \ \  \, y(0)= y_0,\\
&\dot \eta(t) = - \nabla_yV(x(t), y(t)), \quad \eta(0)=\eta_0.
\end{split}\right.
\end{equation*}
This describes the case of two classical point particles interacting with each other via $V(x,y)$. Obviously, if $V(x,y)=V_1(x) +V_2(y)$, the system completely decouples and 
one obtains the dynamics of two independent point particles under the influence of their respective external forces.
\end{remark}

\subsection{The classical limit of the TDSCF system} In this section we shall set $\eps =\delta$ and consider the now fully semi-classically scaled TDSCF system 
where only $0<\eps\ll 1$ appears as a small dimensionless parameter:
\begin{equation}\label{eq:epsTDSCF}
\left\{ \begin{split}
i \eps \partial_t \psi^{\eps}  = \left (-\frac{\eps^2}{2}\Delta_x  + \langle \varphi^{\eps},V\varphi^{\eps} \rangle_{L_y^{2}} \right) \psi^{\eps} \, ,
\quad \psi^{\eps} _{\mid t=0} = \psi^\eps_{\rm in}(x),\\
i\varepsilon \partial_t \varphi^{\eps}  = \left(-\frac{\varepsilon^2}{2}\Delta_y  + \langle \psi^{\eps},h^{\e} \psi^{\eps} \rangle_{L_x^{2}}\right) \varphi^{\eps}  \, ,
\quad \varphi^{\eps}_{\mid t=0} = \varphi^\eps_{\rm in}(y),
\end{split}\right.
\end{equation}
where, as in \eqref{eq:subham}, we denote
\[
h^\eps=-\frac{\eps^2}{2}\Delta_x+V(x,y).
\]
We shall introduce the associated $\e$-scaled Wigner transformations $w^\e[\varphi^\e](y, \eta, t)$ and $W^\e[\psi^\e](x,\xi,t)$ defined by \eqref{wig}. 
From the a-priori estimates established in Lemmas \ref{lem:mass} and \ref{lem:energy}, we infer that both $\psi^\eps$ and $\varphi^{\eps}$ are $\eps$-oscillatory and 
thus we immediately infer the existence of the associated limiting Wigner measures $\mu, \nu \in \mathcal M^+$, such that
$$
w^\e[\varphi^\e ] \stackrel{\e\rightarrow 0_+
}{\longrightarrow}  \mu  \quad \text{in $L^\infty(\R_t;\mathcal S'(\R^d_y \times \R^d_\eta)) \, {\rm weak}^\ast$} ,
$$
and
\[
W^\eps[\psi^\e] \stackrel{\e\rightarrow 0_+
}{\longrightarrow}  \nu \quad \text{in $L^\infty(\R_t;\mathcal S'(\R^d_x \times \R^d_\xi)) \, {\rm weak}^\ast$}.
\]
The associated Wigner transformed system is
\begin{equation}\label{eq:WigTDSCF}
\left\{ \begin{split}
& \partial_t W^{\eps}  + \xi \cdot \nabla_x W^\e + \Theta[\Upsilon^\e]W^\e = 0 \, ,
\quad W^{\eps}_{\mid t=0} = W^\e[\psi^\e_{\rm in}](x,\xi),\\
& \partial _t w^{\e}+\eta \cdot \nabla_y w^{\e} + \Theta [\mathcal V^{\eps}]w^{\varepsilon} =0 ,
\quad\  \quad w^{\eps}_{\mid t=0} = w^\e[\varphi^\e_{\rm in}](y,\eta).
\end{split}\right.
\end{equation}
By following the same arguments as before, we conclude that, up to extraction of sub-sequences, 
\[
\Upsilon^{\eps}(x,t)\stackrel{\eps \rightarrow 0_+}{\longrightarrow}  \iint_{\R^{2n}} V(x,y) \mu( dy, d\eta, t)  \equiv \Upsilon^{}(x,t),
\]
and
\[
\mathcal{V}^{\eps}(y,t)\stackrel{\eps \rightarrow 0_+}{\longrightarrow}  \iint_{\R^{2n}} V(x,y) \nu( dx, d\xi, t)  \equiv \mathcal{V}^{}(y,t), 
\]
on compact sets in $(x,t)$ and $(y,t)$ respectively. Consequently, one can show the convergences of the terms $\Theta[\Upsilon^\e]W^\e$ and $\Theta [\mathcal V^{\eps}]w^{\varepsilon}$ 
by the same techniques as in the proof of Proposition \ref{prop:wignerpotapprox}. In summary, we obtain the following result:
\begin{theorem}\label{thm:classical2} 
Let $V\in \mathcal S(\R_x^{d}\times \R^d_y)$, and $\psi^{\e}\in C(\R_t; H^1(\R^d_x))$, $\varphi^{\e}\in C(\R_t; H^1(\R^n_y))$ be solutions to the system 
\eqref{eq:mixed} with uniformly bounded initial mass and. Then, for any $T>0$, we have that $W^\eps$ and $w^\eps$ converge as $\eps \to 0_+$,  respectively, 
to $\mu \in C(\R_t; \mathcal M^+(\R_y^n\times \R^n_\eta))$ and $\nu \in C(\R_t; \mathcal M^+(\R_x^d\times \R^d_\xi))$, which
solve the classical system \eqref{classicalsystem} in the sense of distributions.
\end{theorem}

In other words, we obtain the same classical limiting system for $\eps=\delta\to 0_+$, as when we took the iterated limit $\eps\to 0_+$ and $\delta \to 0_+$. 
Moreover, it is clear by now that the same result can be achieved from \eqref{eq:TDSCF} by exchanging the role of $\eps$ and $\delta$ and taking the iterated limit 
where first $\delta \to 0_+$ and then $\e \to 0_+$. In summary, we have established the diagram of semi-classical limits as is shown in Figure 1.

\begin{figure}\label{diagram}
\begin{centering}
\includegraphics[scale=0.7]{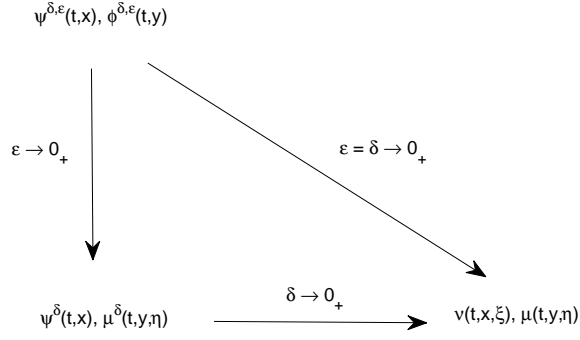}
\caption{The diagram of semi-classical limits: the iterated limit and the classical limit.}
\end{centering}
\end{figure}



\section{Numerical methods based on time-splitting spectral approximations}\label{sec:spectral}

In this section, we will develop efficient and accurate numerical methods for in solving the semi-classically scaled TDSCF equations \eqref{eq:TDSCF} and the 
Ehrenfest equations \eqref{eq:ehrenfest}. 
The highly oscillatory nature of these models strongly suggest the use of spectral algorithms, which are the preferred method of choice 
when dealing with semi-classical models, cf. \cite{JMS}. In the following, we will design and investigate time-splitting spectral algorithms, for both the TDSCF system and 
the Ehrenfest model, which will be shown to be second order in time. The latter is not trivial due to the self-consistent coupling within our equations and it 
will become clear that higher order methods can, in principle, be derived in a similar fashion. 
Furthermore, we will explore the optimal meshing strategy if only physical observables and not the wave function itself are being sought. 
In particular, we will show that one can take time steps independent of semi-classical parameters in order to capture correct physical observables. 

\subsection{The SSP2 method for the TDSCF equations}   

In our numerical context, we will consider the semi-classically scaled TDSCF equations \eqref{eq:TDSCF}
where in one spatial dimension and subject to periodic boundary conditions, i.e.
 \begin{equation}\label{eq:TDSCF1d}
\left\{ \begin{split}
i \delta \partial_t \psi^{\eps,\delta}  = \left (-\frac{\delta^2}{2}\Delta_x  +  \Upsilon^{\e,\delta}(x,t) \right) \psi^{\eps,\delta} \, ,\quad a<x<b\, ,
\quad \psi^{\eps,\delta} _{\mid t=0} = \psi^\delta_{\rm in}(x),\\
i\varepsilon \partial_t \varphi^{\eps,\delta}  = \left(-\frac{\varepsilon^2}{2}\Delta_y  + \Lambda^{\e,\delta}(y,t) \right) \varphi^{\eps,\delta}  \, , \quad a<y<b\, ,
\quad \varphi^{\eps,\delta}_{\mid t=0} = \varphi^\eps_{\rm in}(y),
\end{split}\right.
\end{equation}
subject to
\[
\psi^{\eps}(a,t) = \psi^{\eps}(b,t), \quad \varphi^{\eps}(a,t) = \varphi^{\eps}(b,t), \quad \forall t\in \R.
\] 
As before, we denote $\Upsilon^{\e,\delta} = \langle \varphi^{\e,\delta}, V \varphi^{\e,\delta}\rangle_{L^2_y}$ and $\Lambda^{\e,\delta} = \langle \psi^{\e,\delta}, h^\delta \psi^{\e,\delta}\rangle_{L^2_x}$.

Clearly, $a,b>0$ have to be chosen such that the numerical domain $[a,b]$ is sufficiently large in order to avoid the possible influence of 
boundary effects on our numerical solution. 
The numerical method developed below will work for 
all $\varepsilon$ and $\delta$, even if $\varepsilon=o(1)$ or $\delta=o(1)$. In addition, we will see that it can be naturally extended to the multi-dimensional case.  
			
\subsubsection{The construction of the numerical method}
We assume, on the computational domain $[a,b]$, a uniform spatial grid in $x$ and $y$ respectively,
$x_{j_1}=a+j_1\Delta x$, $y_{j_2}=a+j_2\Delta y$, where $j_{m}=0,\cdots N_{m}-1$, $N_{m}=2^{n_{m}}$, $n_{m}$
are some positive integers for $m=1,2$, and $\Delta x=\frac{b-a}{N_1}$, $\Delta y=\frac{b-a}{N_2}$. We also assume
discrete time $t^{k}=k\Delta t$, $k=0,\cdots,K$ with a uniform time step $\Delta t$. 

The construction
of our numerical method for \eqref{eq:TDSCF1d} is based on the following operator
splitting technique. For every time step $t\in[t^n,t^{n+1}]$, we solve the kinetic step
 \begin{equation}\label{kinetic}
 \left\{ \begin{split}
 i\delta\partial_t \psi^{\varepsilon,\delta} = -\frac{\delta^2}{2}\Delta_{x} \psi^{\varepsilon,\delta},\\
i\varepsilon \partial_t \varphi^{\varepsilon,\delta} = -\frac{\varepsilon^2}{2}\Delta_{y}\varphi^{\varepsilon,\delta}; 
\end{split} \right.
\end{equation} 
and the potential step
 \begin{equation} \label{pot}
  \left\{ \begin{split}
 i\delta\partial_t \psi^{\varepsilon,\delta} = {\Upsilon}^{\varepsilon,\delta}(x,t)\psi^{\varepsilon,\delta},  \\
i\varepsilon\partial_t \varphi^{\varepsilon,\delta} =  {\Lambda}^{\varepsilon,\delta}(y,t)\varphi^{\varepsilon,\delta};
\end{split}\right.
\end{equation} 
possibly for some fractional time steps in a specific order. For example, if Strang's splitting is applied, then the operator splitting error is clearly second order in time (for any fixed value of $\e$). 
However, in the semi-classical regime $\e \to 0_+$, a careful calculation shows that the operator splitting error is actually $O(\Delta t^2/\varepsilon)$, cf. \cite{TS,SL-TS}.

Next, let $U_{j_{1}}^{n}$ be the numerical approximation of the wave functions $\psi^{\e,\delta}$ at $x=x_{j_1}$ and $t=t_{n}$. Then, the kinetic step for $\psi^{\e,\delta}$ 
can be solved {\it exactly} in Fourier space via: 
\begin{equation*}
U_{j_{1}}^{*}=\frac{1}{N_{1}}\sum_{l_{1}=-N_{1}/2}^{N_{1}/2-1}e^{-i\delta \Delta t \mu_{l}^{2}/2} \, \hat{U}_{l_{1}}^{n}e^{i\mu_{l_{1}}(x_{j_{1}}-a)},
\end{equation*}
where $\hat{U}_{l_1}^{n}$ are the Fourier coefficients of $U_{j_1}^{n}$,
defined by
\[
\hat{U}_{l_1}^{n}=\sum_{j_1=0}^{N_1-1}U_{j_1}^{n}e^{-i\mu_{l_1}(x_{j_1}-a)},\quad\mu_{l_1}=\frac{2\pi l_1}{b-a},\quad l_1=-\frac{N_1}{2},\cdots,\frac{N_1}{2}-1.
\]
Similarly, the kinetic step for $\varphi^{\e,\delta}$ can also be solved exactly in the Fourier space.

On the other hand, for the potential step \eqref{pot} with $t_1<t<t_2$, we formally find
\begin{equation}\label{potstep}
\psi^{\varepsilon,\delta}(x,t_2)=\exp\left(-\frac{i}{\delta} \int_{t_1}^{t_2} \Upsilon^{\varepsilon,\delta}(x,s)\, ds \right) \psi^{\varepsilon,\delta}(x,t_1),
\end{equation}
\begin{equation}\label{potstep2}
\varphi^{\varepsilon,\delta}(y,t_2)=\exp\left(-\frac{i}{\e} \int_{t_1}^{t_2} \Lambda^{\varepsilon,\delta}(y,s)\,ds \right) \varphi^{\varepsilon,\delta}(y,t_1),
\end{equation}
where $0<t_2-t_1\leq \Delta t$. The main problem here is, of course, that the mean field 
potentials $\Upsilon^{\varepsilon,\delta}$ and $\Lambda^{\varepsilon,\delta}$ depend on the solution $\psi^{\e,\delta}, \varphi^{\e,\delta}$ themselves. 
The key observation is, however, that within each potential step, the mean field potential $\Upsilon^{\e,\delta}$ is in fact {\it time-independent} (at least if we impose the assumption that 
the external potential $V=V(x,y)$ does not explicitly depend on time). Indeed, a simple calculation shows
\begin{eqnarray*}
\partial_t\Upsilon^{\e,\delta} & \equiv  & \partial_t\left \langle \varphi^{\e,\delta} , V \varphi^{\e,\delta} \right \rangle_{L^2_y} = \left \langle\partial_t\varphi^{\e,\delta}, V  \varphi^{\e,\delta} \right \rangle_{L^2_x}+\left \langle \varphi^{\e,\delta}, V\partial_t\varphi^{\e,\delta} \right \rangle_{L^2_y} \\
&= & \frac{1}{i\varepsilon}\left \langle \varphi^{\e,\delta}, \left( V\Lambda^{\e,\delta}-\Lambda^{\e,\delta}V \right)  \varphi^{\e,\delta} \right \rangle_{L^2_y}   =  0.
\end{eqnarray*} 
In other words, \eqref{potstep} simplifies to 
\begin{equation}\label{psi_for}
\psi^{\varepsilon,\delta}(x,t_2)=\exp\left( -\frac{i(t_1-t_2)}{\delta} \Upsilon^{\varepsilon,\delta}(x,t_1)\right) \psi^{\varepsilon,\delta}(x,t_1).
\end{equation}
which is an {\it exact} solution formula for $\psi^{\varepsilon,\delta}$ at $t=t_2$.

The same argument does not work for the other self-consistent potential $\Lambda^{\e,\delta}$, since formally 
\begin{eqnarray*}
\partial_t\Lambda^{\e,\delta} & \equiv  & \partial_t\left \langle \psi^{\e,\delta}, h^\delta  \psi^{\e,\delta}\right \rangle_{L^2_x}  =  
 \left \langle \partial_t\psi^{\e,\delta} , h^\delta  \psi^{\e,\delta}\right \rangle_x + \left \langle \psi^{\e,\delta} , h^\delta  \partial_t\psi^{\e,\delta}\right \rangle_{L^2_x}  \\
 & = & \frac{1}{i\delta}   \left \langle \psi^{\e,\delta}, \left(  h^\delta \Upsilon^{\e,\delta}-\Upsilon^{\e,\delta} h^\delta \right) \psi^{\e,\delta}\right \rangle_{L^2_x} \\
 & = & \frac{1}{i\delta} \left \langle \psi^{\e,\delta}, -\frac{ \delta^2}{2} \nabla_x \Upsilon^{\e,\delta}\cdot \nabla_x \psi^{\e,\delta}\right \rangle_{L^2_x} +\frac{1}{i\delta} \left \langle \psi^{\e,\delta}, -\frac{ \delta^2}{2} \Delta_x \Upsilon^{\e,\delta}\psi^{\e,\delta}\right \rangle_{L^2_x}\\
 & = & \frac{1}{2}\left \langle \psi^{\e},  \nabla_x \Upsilon^{\e}\cdot \left( i\delta \nabla_x \right)  \psi^{\e} \right \rangle_{L^2_x} + \frac{i \delta}{2} \left \langle \psi^{\e},  \Delta_x \Upsilon^{\e} \psi^{\e} \right \rangle_{L^2_x}.
\end{eqnarray*}
However, since $\Lambda^{\e,\delta}(y,t)=\langle \psi^{\eps,\delta},h^{\delta} \psi^{\eps,\delta} \rangle_{L_x^{2}}$, 
the formula \eqref{psi_for} for $\psi^{\e,\delta}$ allows to evaluate $\Lambda^{\e,\delta}(y,t)$ for any $t_1<t<t_2$. Moreover, 
the above expression for $\partial_t \Lambda^{\e,\delta}$, together with the Cauchy-Schwarz inequality and the energy estimate in Lemma \ref{lem:energy}, directly implies 
that $\partial_t \Lambda^{\e,\delta}$ is $O(1)$. 
Thus, one can use standard numerical integration methods to approximate the time-integral within \eqref{potstep2}. For example, one can use the trapezoidal rule to obtain
\begin{equation}\label{phi_approx}
\varphi^{\varepsilon,\delta}(y,t_2) \approx \exp\left(-{\frac{i( \Lambda^{\e,\delta}(y,t_2)+\Lambda^{\e,\delta}(y,t_1))(t_1-t_2)}{2\eps}}\right)\varphi^{\varepsilon,\delta}(y,t_1). 
\end{equation}
Obviously, this approximation introduces a phase error of order $O(\Delta t^2/\varepsilon)$, which is comparable to the operator splitting error. 
This is the reason why we call the outlined numerical method SSP2, i.e., a second order Strang-spliting spectral method.

\begin{remark}
In order to obtain a higher order splitting method to the equations, one just needs to use a higher order quadrature rule to approximate the time-integral within \eqref{potstep2}.
\end{remark}

\subsubsection{Meshing strategy}

In this subsection, we will analyze the dependence on the semi-classical parameters of the numerical error by applying the Wigner transformation onto the scheme we proposed above. 
In particular, this yields an estimate on the approximation error for (the expectation values of) physical observables due to \eqref{eq:expectation}. 
Our analysis thereby follows along the same lines as in Refs. \cite{TS,SL-TS}. For the sake of simplicity, we shall only consider the differences between their cases and ours.

We denote the Wigner transforms $W^{\e,\delta}\equiv W^{\delta}[\psi^{\e,\delta}]$ and $w^{\e,\delta}=w^{\e}[\varphi^{\e,\delta}]$, which satisfy the system
\begin{equation}\label{eq:WigTDSCF2}
\left\{ \begin{split}
& \partial_t W^{\eps,\delta}  + \xi \cdot \nabla_x W^{\e,\delta} + \Theta[\Upsilon^{\e,\delta}]W^{\e,\delta} = 0 \, ,
\quad W^{\eps,\delta}_{\mid t=0} = W^{\delta}[\psi^{\delta}_{\rm in}](x,\xi),\\
& \partial _t w^{\e,\delta}+\eta \cdot \nabla_y w^{\e,\delta} + \Theta [\mathcal V^{\eps,\delta}]w^{\varepsilon,\delta} =0 ,
\quad\  \quad w^{\eps,\delta}_{\mid t=0} = w^{\e}[\varphi^\e_{\rm in}](y,\eta).
\end{split}\right.
\end{equation} 
Clearly, the time splitting for the Schr\"odinger equation induces an analogous time-splitting of the associated the Wigner equations \eqref{eq:WigTDSCF2}. 
Having in mind the properties of the SSP2 method, we only need to worry about the use of the the trapezoidal rule in approximating $\varphi^{\e,\delta}$ within the potential step. 
We shall consequently analyze the error induced by this approximation in the computation of the Wigner transform. 
To this end, we are interested in analyzing two special cases: $\delta=O(1)$, and $\e \ll 1$,  or $\delta=\e \ll 1$. These correspond to the semi-classical limits 
we showed in Theorem \ref{thm:mixed} and Theorem \ref{thm:classical2}.  

We first consider the case $\delta=\e \ll 1$, where Wigner transformed TDSCF system reduces to \eqref{eq:WigTDSCF}. 
In view of (\ref{phi_approx}), if we denote the approximation on the right hand side by $\tilde{\varphi}^{\e}$, then $\tilde {\varphi}^{\e}$ is the exact solution to the following equation
\begin{equation*}
i\varepsilon \partial_t { \varphi^\e} =  G(y)\tilde{\varphi^\e},\quad t_1< t < t_2, 
\end{equation*} 
where
\[
G^{\e}(y)=\frac{1}{2} (\Lambda^{\varepsilon}(y,t_1)+\Lambda^{\varepsilon}(y,t_2)).
\]

If one denotes the Wigner transform of $\tilde{\varphi^\e}(y,t)$ by $\tilde{w^\e}(y,\eta,t)$, then, by the same techniques as in the previous sections, 
one can show that $\tilde{w^\e}$ satisfies
\begin{equation} \label{tildew}
\partial_t \tilde{w^\e}-\nabla_y G^{\e}\cdot \nabla_{\eta}\tilde{w^\e}+O(\varepsilon)=0.
 \end{equation}
In order to compare $w^{\e}(y,\eta,t_2)$ and $\tilde{w}(y,\eta,t_2)$, we now consider the following set of equations
\begin{align*}
& \partial_t w_1  - \nabla_y \mathcal V^{\varepsilon}(y,t)\cdot \nabla_\eta w_1=0, \quad  t_1< t < t_2, \\
& \partial_t w_2  - \nabla_y G^{\e}(y)\cdot \nabla_\eta w_2=0, \quad t_1< t < t_2, 
\end{align*}
subject to the same initial condition at $t=t_1$:
\[
 w_1(y,\eta,t_1)= w_2(y,\eta,t_1)= w_0(y,\eta).
\]
By the trapezoidal rule, 
\[
\int_{t_1}^{t_{2}}\nabla_y \mathcal V^{\varepsilon}(y,s) \, ds \approx (t_2-t_1)\nabla_y G^\e(y),
\] 
since $\nabla_y \Lambda^\e(y,t)\equiv \nabla_y \mathcal V^\e (y,t).$ Thus, by the method of characteristics, it is straightforward to measure the discrepancy between $w_1$ and $w_2$ at $t=t_2$
and one easily obtains 
\[
w_1-w_2=O \left(  \Delta t ^3 \right).
\]
Furthermore, within the potential step of the time-split Wigner equation, equation \eqref{tildew} together with the method of characteristics, implies at $t=t_2$
\begin{equation}\label{drop_eps}
\tilde{w^{\varepsilon}}-w_1=O \left( \varepsilon \Delta t  \right), \quad \tilde{w^\e}-w_2=O \left(  \varepsilon \Delta t \right), 
\end{equation}
where we have used $0<t_2-t_1\leq \Delta t$. 
  
In summary, we conclude that for the SSP2 method, the approximation within the potential step results in an one-step error which is bounded by $ O(\varepsilon \Delta t + \Delta t ^3 ) $. 
Thus, for fixed $\Delta t$, and as $\varepsilon \rightarrow 0_+$, this one-step error in computing the physical observables is dominated by $O(\Delta t^3)$ and we consequently can 
take $\varepsilon$-{\it independent} time steps for accurately computing semi-classical behavior of physical observables. 
By standard numerical analysis arguments, cf. \cite{TS,SL-TS}, one consequently finds, that the SSP2 method introduces an $O(\Delta t^2)$ error in computing the physical 
observables for  $\e \ll 1$ within an $O(1)$ time interval. Similarly, one can obtain the same results when $\delta$ is fixed while $\e \ll 1$.

We remark that, if a higher order operator splitting is applied to the TDSCF equations, and if a higher
order quadrature rule is applied to approximate formula \eqref{potstep2}, one obviously can expect higher order convergence in the physical observables.


\subsection{ The SVSP2 method for the Ehrenfest equations}  

In this section, we consider the one-dimensional Ehrenfest model obtained in Section \ref{sec:ehrenfest}. More precisely, we  consider 
a (semi-classical) Schr$\ddot{\rm o}$dinger equation 
coupled with Hamilton's equations for a classical point particle, i.e 
\begin{equation}\label{eq:ehrenfest2}
\left\{ \begin{split}
& i \delta \partial_t \psi^{\delta}  = \left (-\frac{\delta^2}{2}\Delta_x  + V(x,y(t)) \right) \psi^{\delta} \, ,\,\quad a<x<b\,,
 \\
&\dot y(t) = \eta(t),\quad \dot \eta(t)= - \int_{\R^d} \nabla_y V\left(x,y(t)\right) |\psi^{\delta}(x,t)|^2\, dx, 
\end{split}\right.
\end{equation}
with initial conditions
\[
\psi^{\e} _{\mid t=0} = \psi^\e_{\rm in}(x), \quad y_{\mid t=0} =y_0, \quad \eta_{\mid t=0}= \eta_0,
\]
and subject to periodic boundary conditions.
Inspired by the SSP2 method, we shall present a numerical method to solve \eqref{eq:ehrenfest2}, 
which is second order in time and works for all $0<\delta \le 1 $. 

As before, we assume a uniform spatial grid
$x_{j}=a+j\Delta x$, where $N=2^{n_{0}}$, $n_{0}$
is an positive integer and $\Delta x=\frac{b-a}{N}$. We also assume
uniform time steps $t^{k}=k\Delta t$, $k=0,\cdots,K$ for both the Schr$\ddot{\rm o}$dinger equation and Hamilton's ODEs. 
For every time step $t\in[t^n,t^{n+1}]$, we split the system \eqref{eq:ehrenfest2} into a kinetic step
 \begin{equation} \label{ehrenfest_kinetic}
 \left\{ \begin{split}
& i\delta \partial_t  \psi^{\delta}(x,t) = -\frac{\delta^2}{2}\Delta_{x} \psi^{\delta}(x,t), \\
& \dot{y}=\eta,\quad \dot{\eta}=0; 
\end{split}\right.
\end{equation} 
and a potential step
 \begin{equation} \label{ehrenfest_pot}
  \left\{ \begin{split}
& i\delta \partial_t  \psi^{\delta}(x,t) =  V(x,y(t)) \psi^{\delta} (x,t), \\
& \dot{y}=0,\quad \dot{\eta}=-\int_{\R^d} \nabla_y V\left(x,y(t)\right) |\psi^{\delta}(x,t)|^2\, dx.
\end{split}\right.
\end{equation} 
We remark that, the operator splitting method for the Hamilton's equations may
be one of the symplectic integrators. The readers may refer to \cite{OdeSp} for a systematic discussion. 

As before, the kinetic step \eqref{ehrenfest_kinetic} can be solved analytically. 
On the other hand, within the potential step \eqref{ehrenfest_pot}, we see that
\begin{equation}\label{ydot}
\partial_t V(x,y(t))  =  \nabla_y V \cdot \dot{y}(t)=0,
\end{equation}
i.e., $V(x,y(t))$ is indeed time-independent.  Moreover
\begin{align*}
& \partial_t \left( \int_{\R^d} \nabla_y V\left(x,y(t)\right) |\psi^{\delta}(x,t)|^2\, dx \right) \\
 & =  \left \langle {\partial_t} \psi^{\delta}, \nabla_y V(x,y(t))  \psi^{\delta} \right \rangle_{L^2_x}+\left \langle \psi^{\delta}, \nabla_y V(x,y(t)) {\partial_t} \psi^{\delta} \right \rangle_{L^2_x}  \\
& \quad + \left \langle \psi^{\delta}, {\partial_t} \nabla_y V(x,y(t))  \psi^{\delta} \right \rangle_{L^2_x}
\end{align*}
Now, we can use the first equation in \eqref{ehrenfest_pot} and the fact that $V(x,y(t))\in \R$ to infer that the first two terms on the right hand side of this time-derivate cancel each other. We thus have 
\begin{align*}
\partial_t \left( \int_{\R^d} \nabla_y V\left(x,y(t)\right) |\psi^{\delta}(x,t)|^2\, dx \right) =   \left \langle \psi^{\delta}, \nabla^2_y V(x,y(t))\cdot \dot{y}(t) \psi^\delta \right \rangle_{L^2_x}  =  0,
\end{align*}
in view of \eqref{ydot}. In other words, also the semi-classical force is time-independent within each potential step. In summary, we find that for $t\in[t_1,t_2]$, the 
potential step admits the following exact solutions  
\[
\psi^{\delta}(x,t_2)=\exp\left(\frac{i}{\delta}(t_1-t_2)V(x,y(t_1)) \right) \psi^{\delta}(x,t_1),
\]
as well as 
\[
y(t_2)=y(t_1),\quad \eta(t_2)=\eta(t_1)-(t_2-t_1)\int_{\R^d} \nabla_y V\left(x,y(t_1)\right) |\psi^{\delta}(x,t_1)|^2\, dx .
\]
This implies, that for this type of splitting method, there is {\it no numerical error in time} within the kinetic or the potential steps and thus, we only 
pick up an error of order $O(\Delta t^2/\delta)$ in the wave function and an error of order $O(\Delta t^2)$ in the classical coordinates induced by the operator splitting. 
Standard arguments, cf. \cite{TS,SL-TS}, then imply that one can use $\delta$-independent time steps to correctly capture the expectation values of physical observables. 
We call this proposed method SVSP2, i.e., a second order Strang-Verlet splitting spectral method. It is second order in time but can easily be improved by using higher order operator splitting 
methods for the Schr\"odinger equation and for Hamilton's equations.


\section{Numerical tests}\label{sec:tests}
In this section, we test the SSP2 method for the TDSCF equations and the SVSP2 method for the Ehrenfest system.  
In particular, we want to test the methods after the formation of caustics, which generically appear in the WKB approximation of the Schr\"odinger wave functions, cf \cite{SMM}.
We also test the convergence properties in time and with respect to the spatial grids for the wave functions and the following physical observable densities 
\[
\rho^\e(t,x) = |\psi^\e(t,x)|^2, \quad J^\e(t,x) = \e \text{Im} (\overline{\psi ^\e}(x,t) \nabla \psi^\e(x,t)),
\]
i.e., the particle density and current densities associated to $\psi^\e$ (and analogously for $\varphi^\e$).

\subsection{SSP2 method for the TDSCF equations} We first study the behavior of the proposed SSP2 method. 
In Example 1, we fix $\delta$ and test the SSP2 method for various $\e$. In Example 2 and Example 3,
 we take $\delta=\eps$ and assume the same spatial grids in $x$ and $y$. 
\medskip

\textbf{Example 1.} In this example, we fix $\delta=1$, and test the SSP2 method for various $\e=o(1)$. 
We want to test the convergence in spatial grids and time, and whether $\e$-independent time steps can be taken to calculate accurate physical observables.  

Assume $x,y\in[-\pi,\pi]$ and let $V(x,y)=\frac12(x+y)^2$. The initial conditions are of the WKB form,
$$
\psi^{\delta}_{\rm in}(x)=e^{-2(x+0.1)^2}e^{i\sin{x}/\delta}, \quad \varphi^{\e}_{\rm in}(y)=e^{-5(y-0.1)^2}e^{i\cos{y}/\varepsilon}.
$$
In the following all our numerical tests are computed until a stopping time $T=0.4$.

We first test the convergence of the SSP2 method in $\Delta x$ and $\Delta y$, respectively. By the energy estimate in Lemma 2.4, 
one expects the meshing strategy $\Delta x= O(\delta)$ and $\Delta y=O(\e)$ to obtain spectral accuracy. 
We take $\delta=1$ and $\eps=\frac{1}{1024}$. The reference solution is computed with sufficiently fine spatial grids and time steps: 
$\Delta x=\Delta y= \frac{2\pi}{32768}$ and $\Delta t=\frac{0.4}{4096}$. We repeated the tests with the same $\Delta y$ and $\Delta t$ 
but different $\Delta x$, or with the same $\Delta x$ and $\Delta t$ but different $\Delta y$. 
The errors in the wave functions and the position densities are calculated and plotted in Figure 2, from which we observe clearly that 
$\Delta x=O(1)$ and $\Delta y=O(\e)$ are sufficient to obtain spectral accuracy. 
Due to the time discretization error, the numerical error cannot be reduced further once $\Delta x$ and $\Delta y$ become sufficiently small. 

\begin{figure}
\begin{centering}
\includegraphics[scale=0.6]{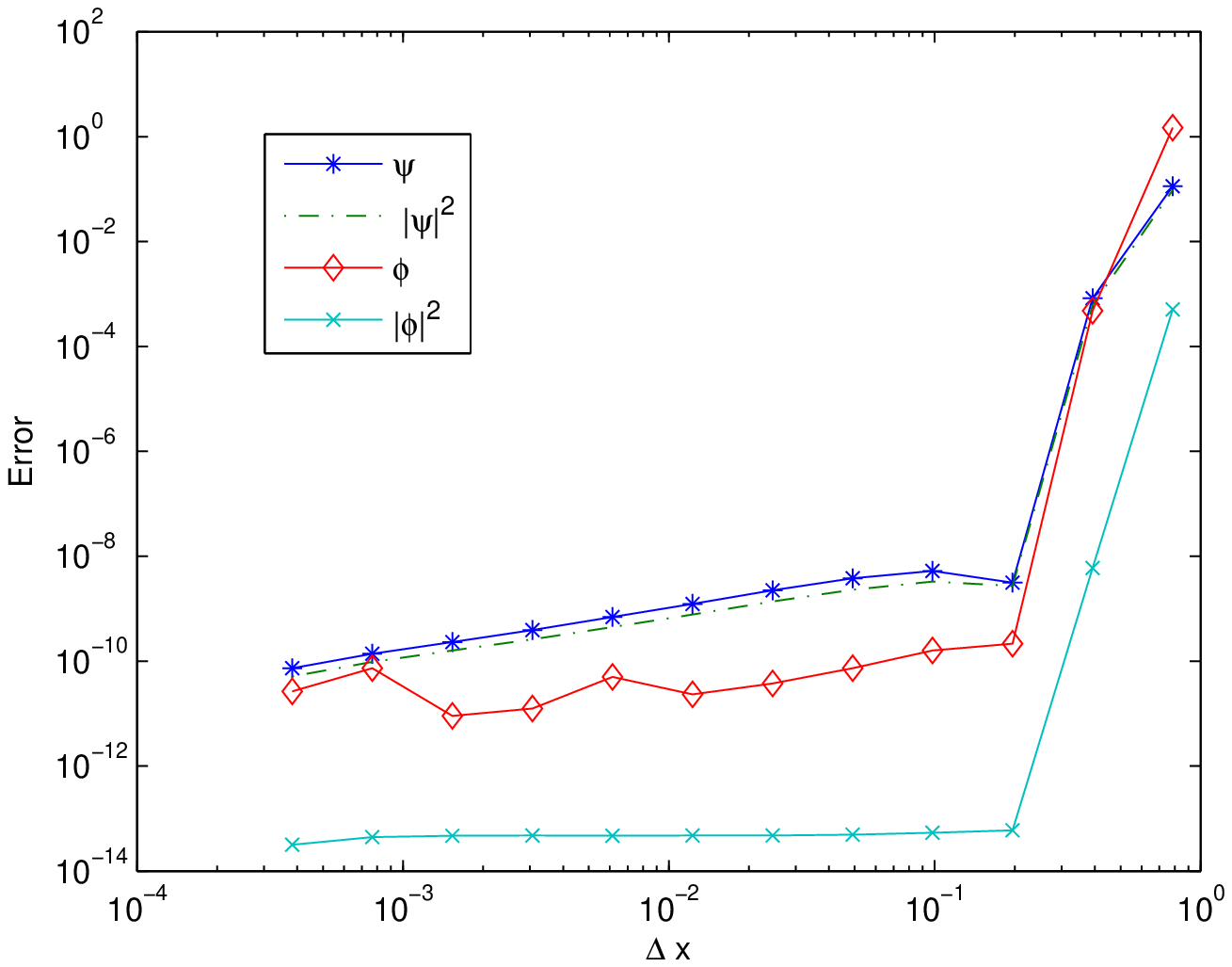} \par \includegraphics[scale=0.6]{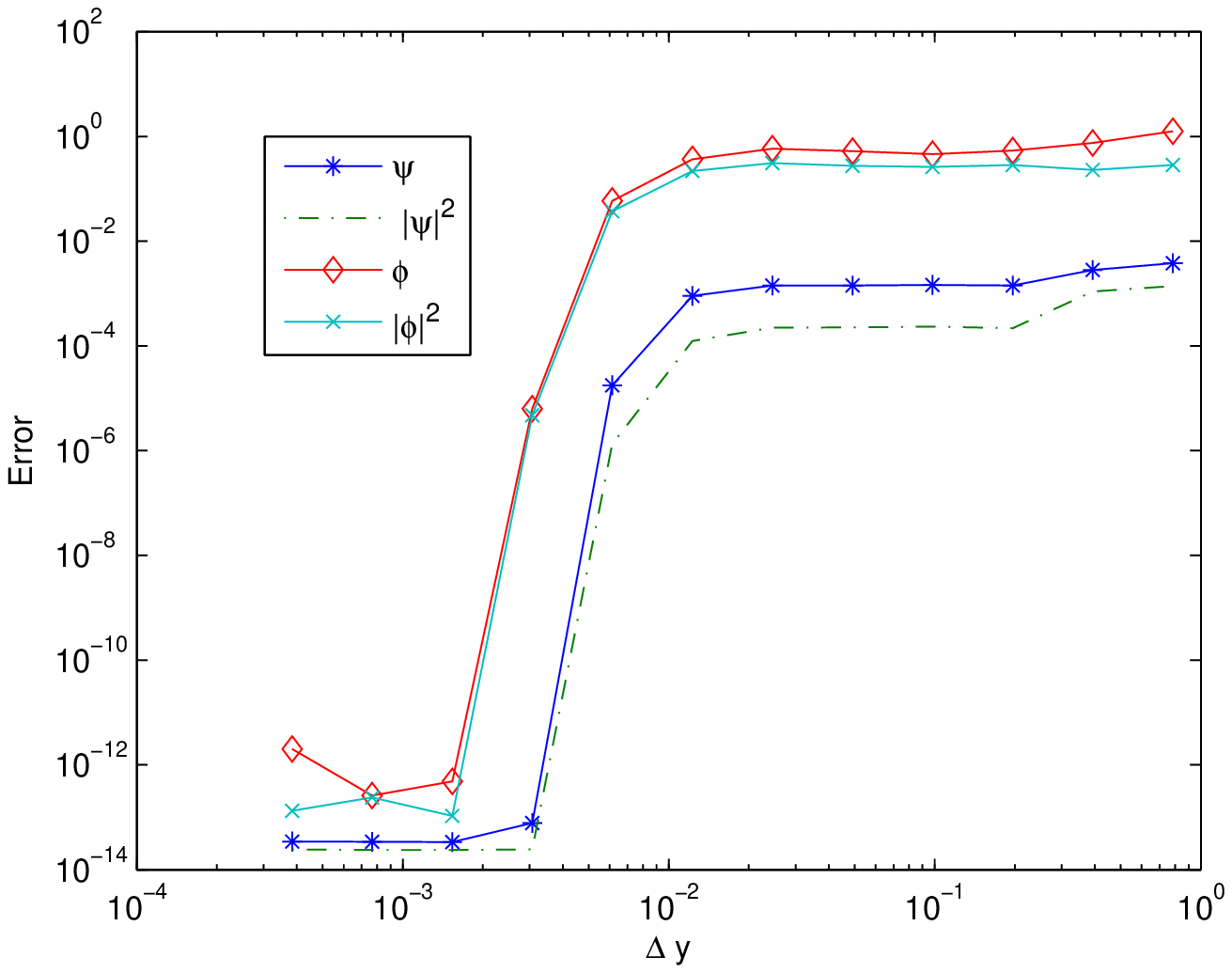} \par
\caption{Reference solution: $\Delta x=\Delta y= \frac{2\pi}{32768}$ and $\Delta t=\frac{0.4}{4096}$. Upper picture: fix $\Delta y= \frac{2\pi}{32768}$ and $\Delta t=\frac{0.4}{4096}$, take $\Delta x=\frac{2\pi}{16384}$, $\frac{2\pi}{8192}$, $\frac{2\pi}{4096}$, $\frac{2\pi}{2048}$, $\frac{2\pi}{1024}$, $\frac{2\pi}{512}$, $\frac{2\pi}{256}$, $\frac{2\pi}{128}$, $\frac{2\pi}{64}$, $\frac{2\pi}{32}$, $\frac{2\pi}{16}$, $\frac{2\pi}{8}$. Lower Picture: fix $\Delta x= \frac{2\pi}{32768}$ and $\Delta t=\frac{0.4}{4096}$, take $\Delta y=\frac{2\pi}{16384}$, $\frac{2\pi}{8192}$, $\frac{2\pi}{4096}$, $\frac{2\pi}{2048}$, $\frac{2\pi}{1024}$, $\frac{2\pi}{512}$, $\frac{2\pi}{256}$, $\frac{2\pi}{128}$, $\frac{2\pi}{64}$, $\frac{2\pi}{32}$, $\frac{2\pi}{16}$, $\frac{2\pi}{8}$.}
\end{centering}
\end{figure}

Next, to test the the convergence in time, we take $\delta=1$, $\varepsilon=\frac{1}{1024}$, and compare to a reference solution which is computed through a well resolved mesh with 
 $\Delta x=\frac{2\pi}{512}$, $\Delta y= \frac{2\pi}{16348}$ and $\Delta t=\frac{0.4}{4096}$. 
Then, we compute with the same spatial grids, but with different time steps. 
The results are illustrated in the Figure 3. We observe that the method is stable even if $\Delta t \gg \varepsilon$. Moreover, 
we get second order convergence in the wave functions as well as in the physical observable densities. 

\begin{figure}
\begin{centering}
\includegraphics[scale=0.6]{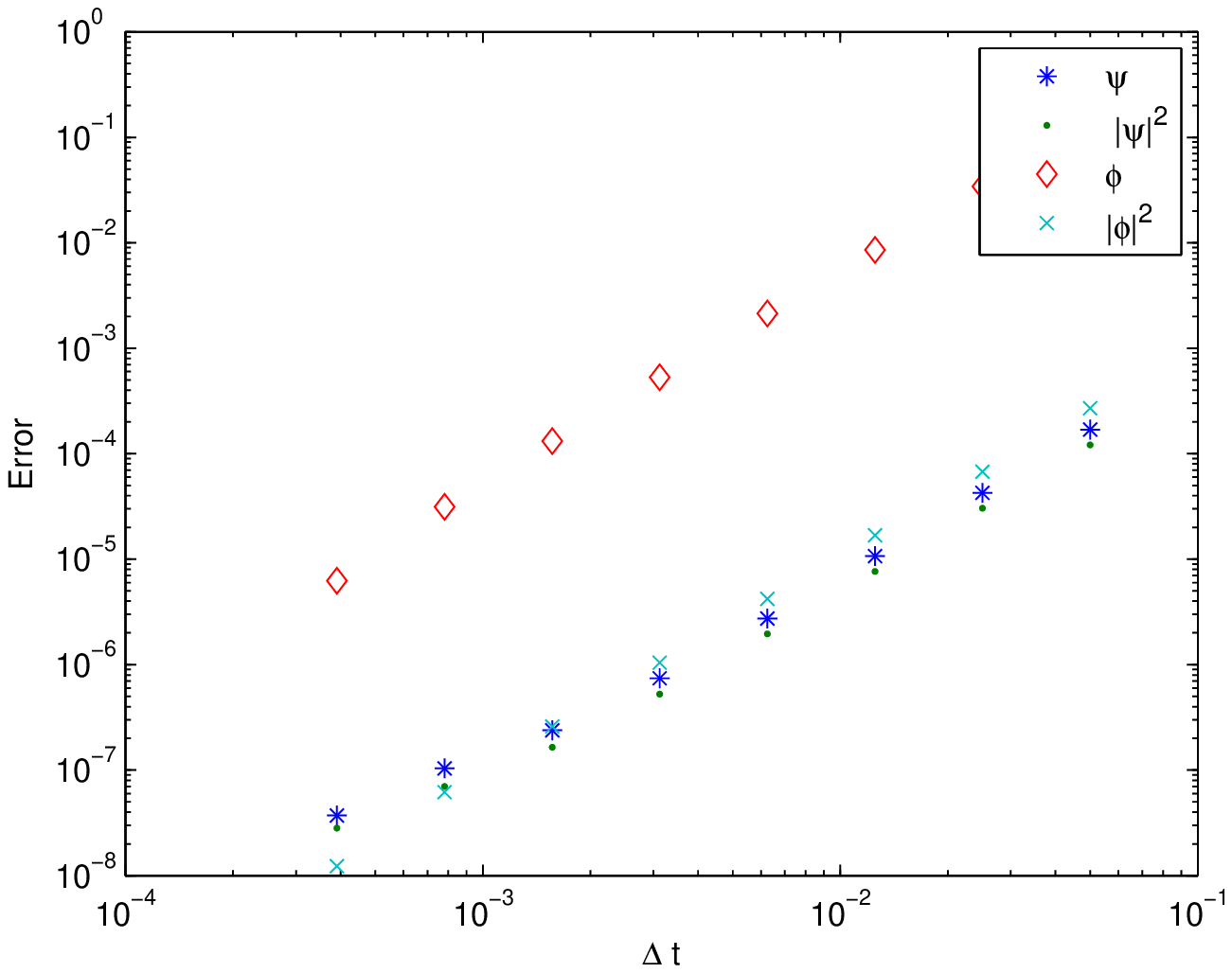}  \par
\caption{Reference solution: $\Delta x=\frac{2\pi}{512}$, $\Delta y= \frac{2\pi}{16348}$ and $\Delta t=\frac{0.4}{4096}$. SSP2: fix $\Delta x=\frac{2\pi}{512}$, $\Delta y= \frac{2\pi}{16348}$, take $\Delta t=\frac{0.4}{1024}$, $\frac{2\pi}{512}$, $\frac{2\pi}{256}$, $\frac{2\pi}{128}$, $\frac{2\pi}{64}$, $\frac{2\pi}{32}$, $\frac{2\pi}{16}$, $\frac{2\pi}{8}$}.
\end{centering}
\end{figure}

At last, we test whether $\eps$-independent $\Delta t$ can be taken to capture the correct physical observables. 
We solve the TDSCF equations with resolved spatial grids. The numerical solutions with $\Delta t=O(\eps)$ 
are used as the reference solutions. For $\eps=\frac{1}{64}$, $\frac{1}{128}$, $\frac{1}{256}$, $\frac{1}{512}$, $\frac{1}{1024}$, $\frac{1}{2048}$ and $\frac{1}{4096}$, we fix $\Delta t=\frac{0.4}{8}$. 
The errors in the wave functions and position densities are calculated. 
We see in Figure 4 that, the error in the wave functions increases as $\e \rightarrow 0_{+}$, but the error in physical observables does not change notably.

\begin{figure}
\begin{centering}
\includegraphics[scale=0.6]{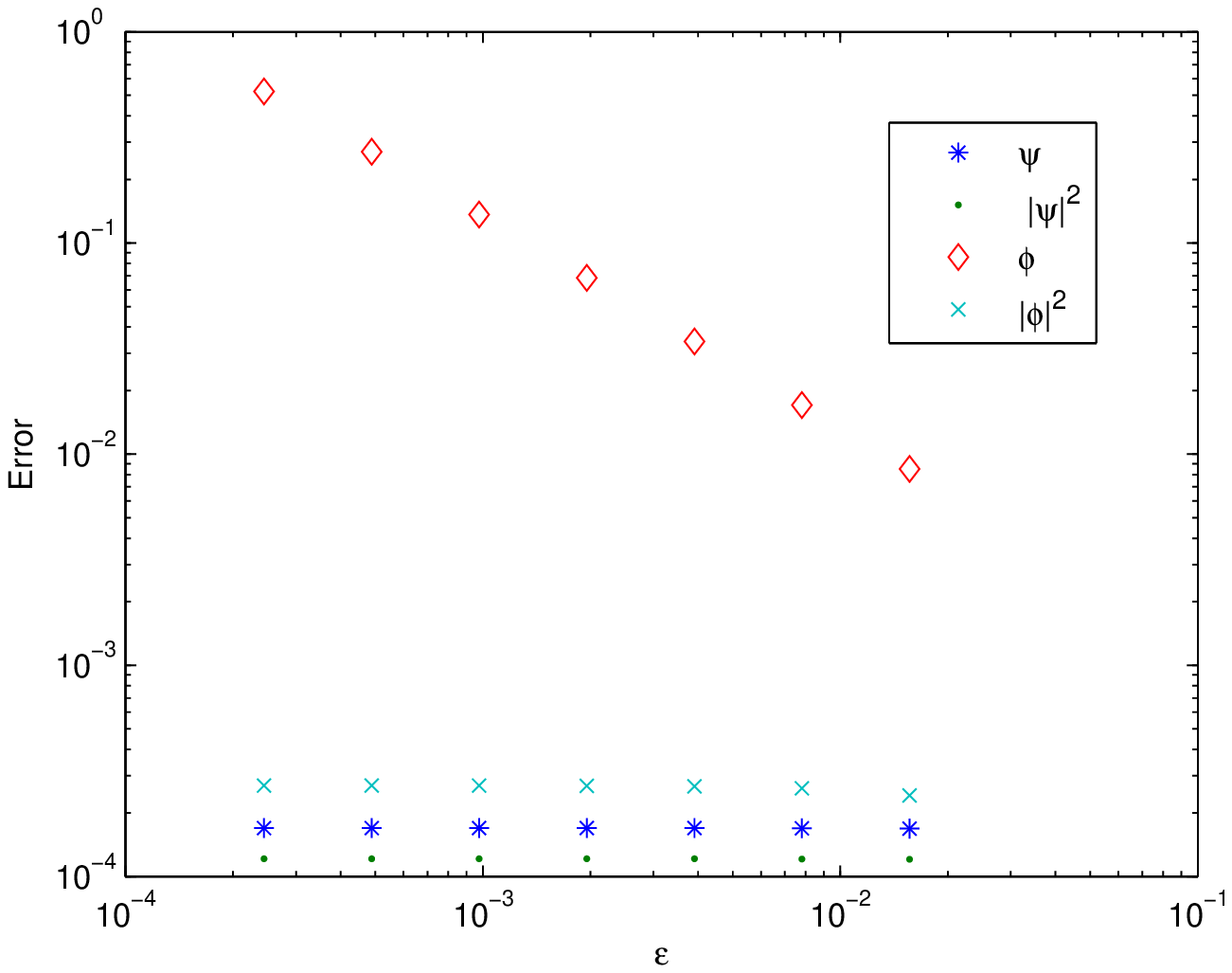} \par
\caption{Fix $\Delta t=0.05$. For $\eps=1/64$, $1/128$, $1/256$, $1/512$, $1/1024$, $1/2048$ and $1/{4096}$, 
$\Delta x=2\pi\varepsilon/16$, respectively. The reference solution is computed with the same $\Delta x$, but $\Delta t={\eps}/{10}$.}
\end{centering}
\end{figure}

\medskip

\textbf{Example 2.} We want to numerically verify the behavior of the TDSCF system as $\eps=\delta \to 0_+$ compared to the classical limit. To this end, let $x,y \in [0,1]$, and assume periodic boundary conditions for both equations. 
Assume $V(x,y)=1$, and choose initial conditions of WKB form
\begin{align*}
& \psi^{\e}_{\rm in}(x)=e^{-25(x-0.58)^2}e^{-i\ln{\left(2\cosh{5(x-0.6)}\right)}/5\e}, \\
& \varphi^\e_{\rm in}(y)=e^{-25(y-0.5)^2}e^{-i\ln{\left(2\cosh{5(y-0.5)}\right)}/5\varepsilon}.
\end{align*}
The tests are done for $\varepsilon=\frac{1}{512}$ and $\varepsilon=\frac{1}{2048}$, respectively. 
Note that, the potential $V$ is chosen in this simple form so that the semi-classical limit can be computed analytically. Indeed, the classical limit yields a decouples system of two independent Vlasov equations,  similar to the examples in \cite{TS,SL-TS,Cau2}. The formation of caustics was previously analyzed in \cite{Cau1, Cau2} and it is known that the caustics is formed for $t<0.54$.

We solve the TDSCF equations by the SSP2 method until $T=0.54$ with two different meshing strategies
\[
\Delta x= O(\varepsilon), \quad \Delta t= O(\varepsilon);
\]
and
\[
\Delta x= O(\varepsilon), \quad \Delta t= o(1).
\]
The numerical solutions are then compared with the semi-classical limits: 
In Figure 5 and Figure 6, the dashed line represents the semi-classical limits (\ref{classicalsystem}), the dotted line represents the 
numerical solution with $\varepsilon$-independent $\Delta t$, and the solid line represents the numerical solution with $\varepsilon$-dependent $\Delta t$. 
The figures confirm that the semi-classical limits are still valid after caustics formation, and 
that the numerical scheme can capture the physical observables with $\varepsilon$-independent $\Delta t$.

\begin{figure}\label{512limit}
\begin{centering}
\includegraphics[scale=0.40]{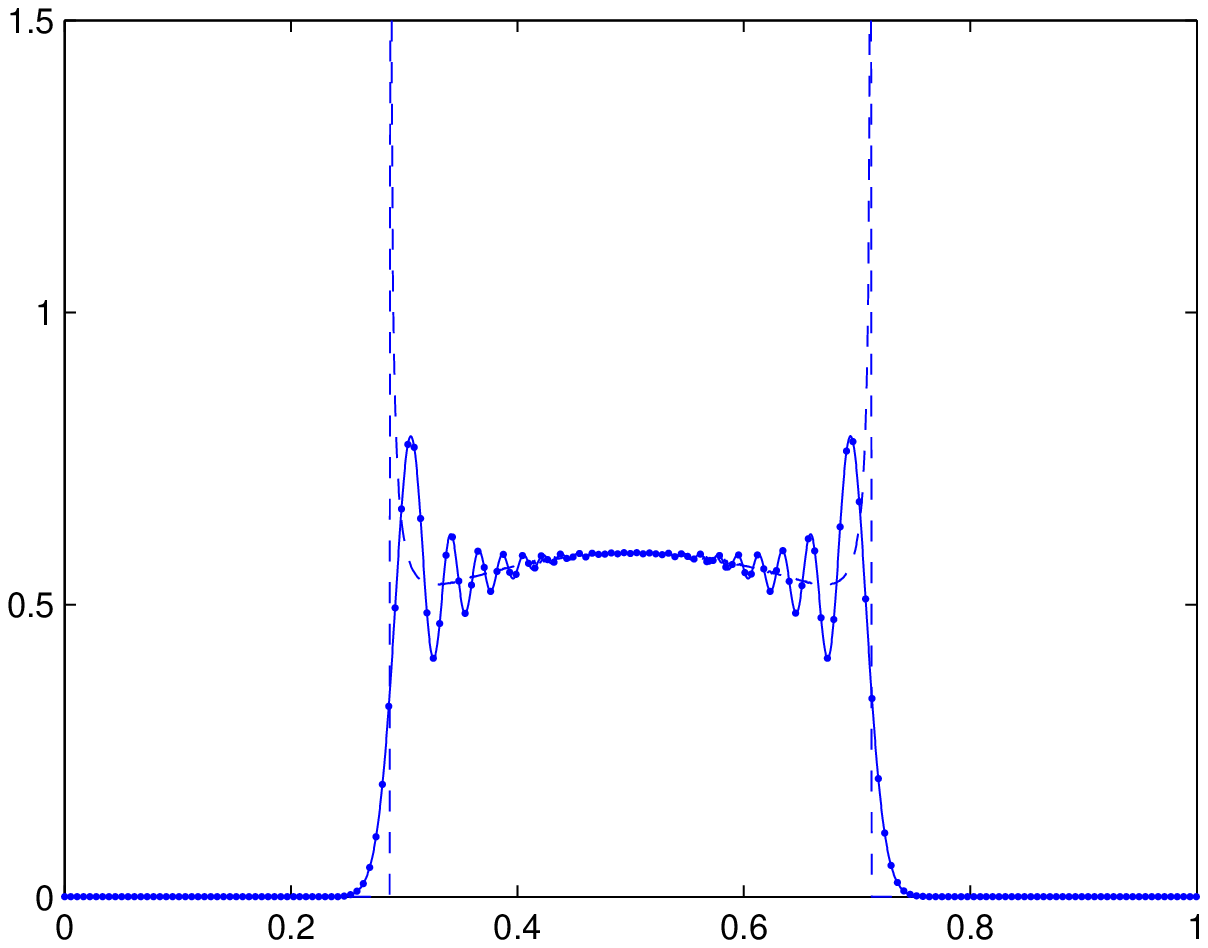}
\includegraphics[scale=0.40]{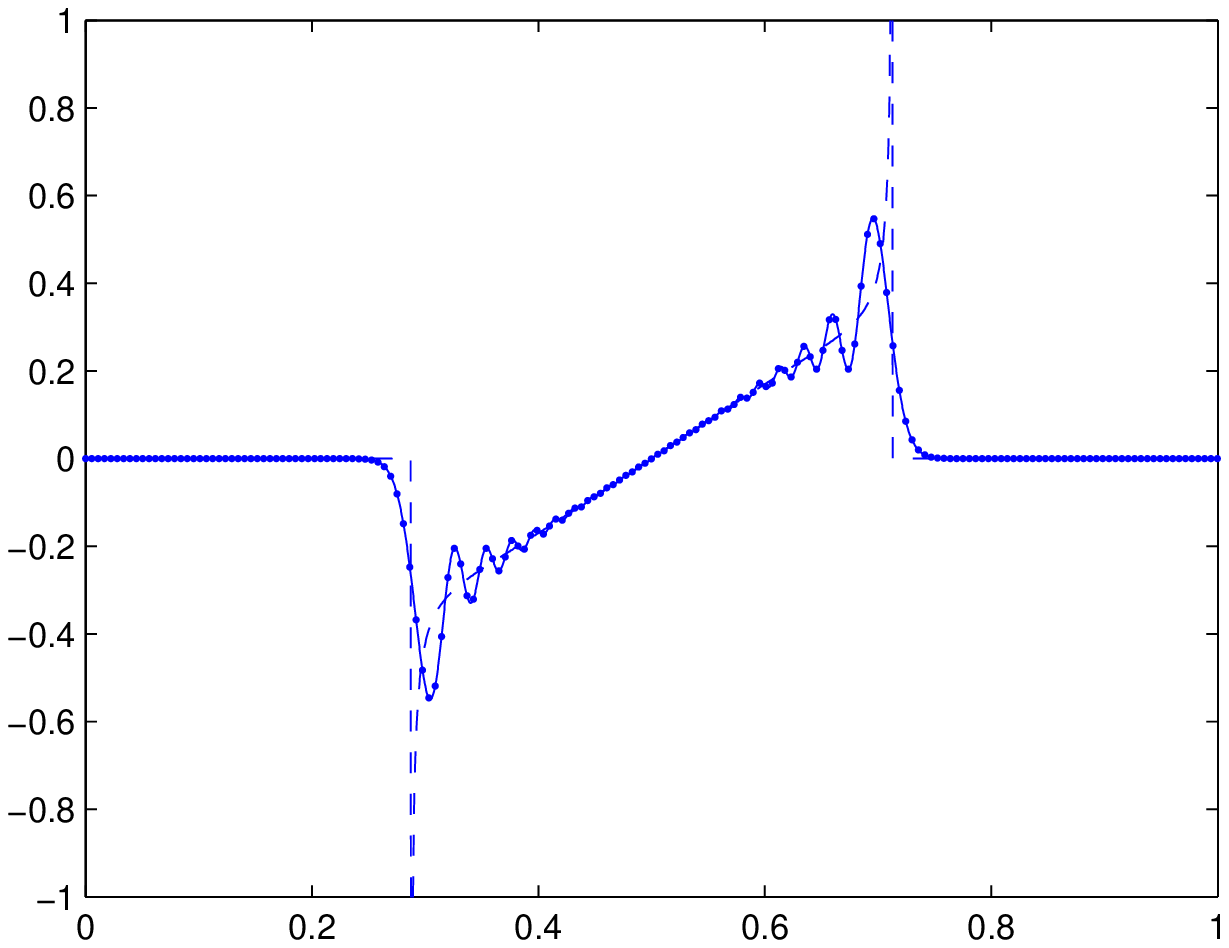}
\par \end{centering}
\begin{centering}
\includegraphics[scale=0.40]{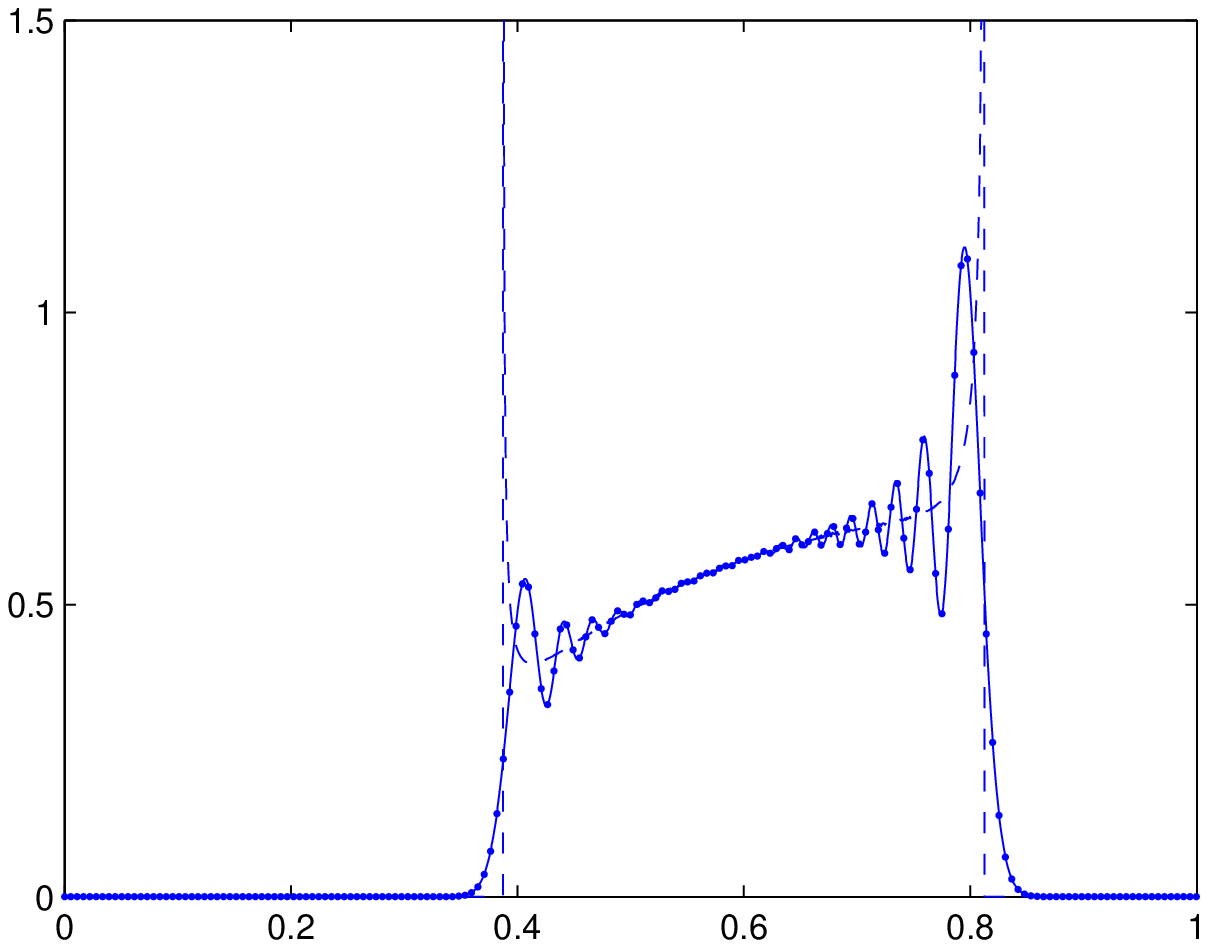}
\includegraphics[scale=0.40]{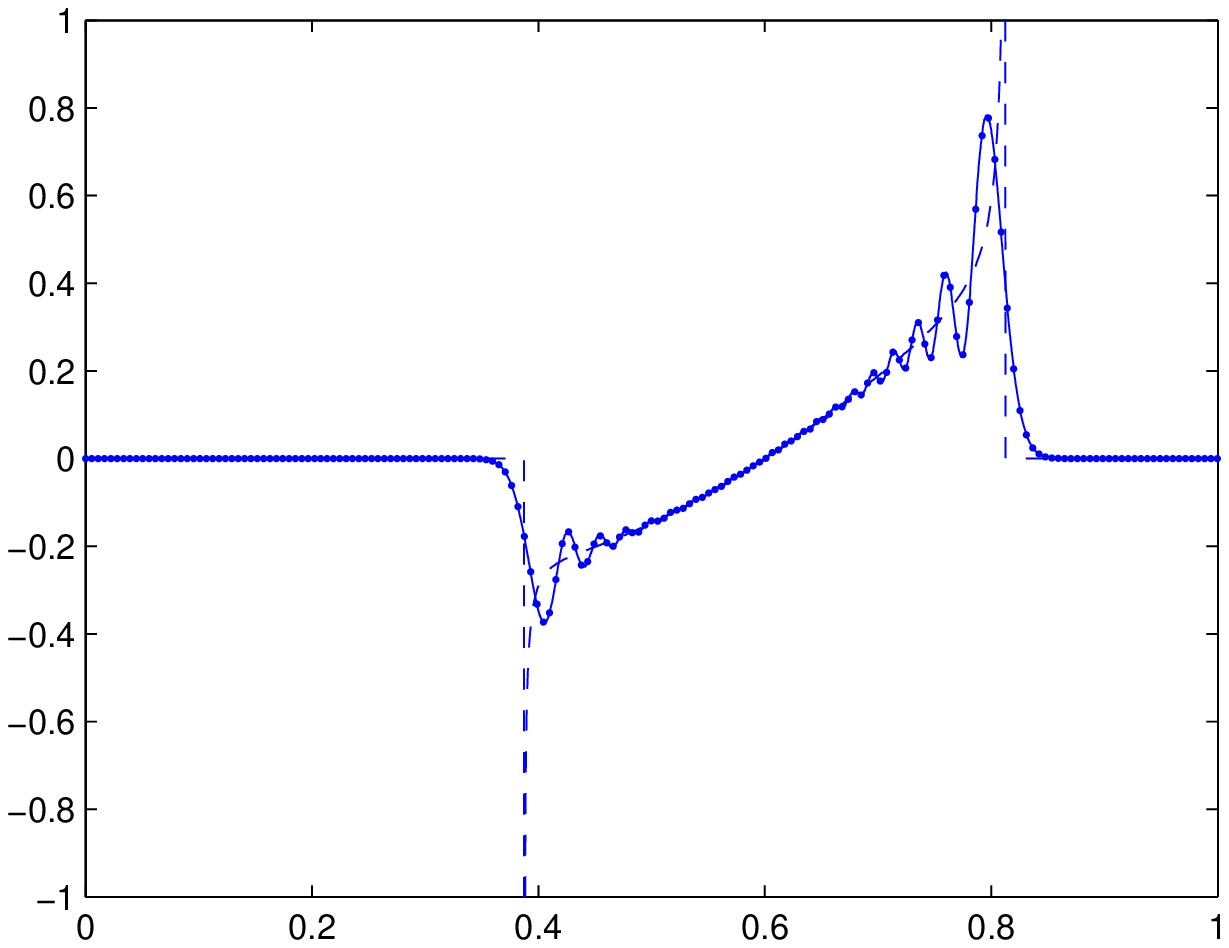}
\par \end{centering}
\caption{$\varepsilon=\frac{1}{512}$. First row: position density and current density of $\varphi^{\e}$; second row: position density and flux density of $\psi^{\e}$.}
\end{figure}
\begin{figure}\label{2048limit}
\begin{centering}
\includegraphics[scale=0.40]{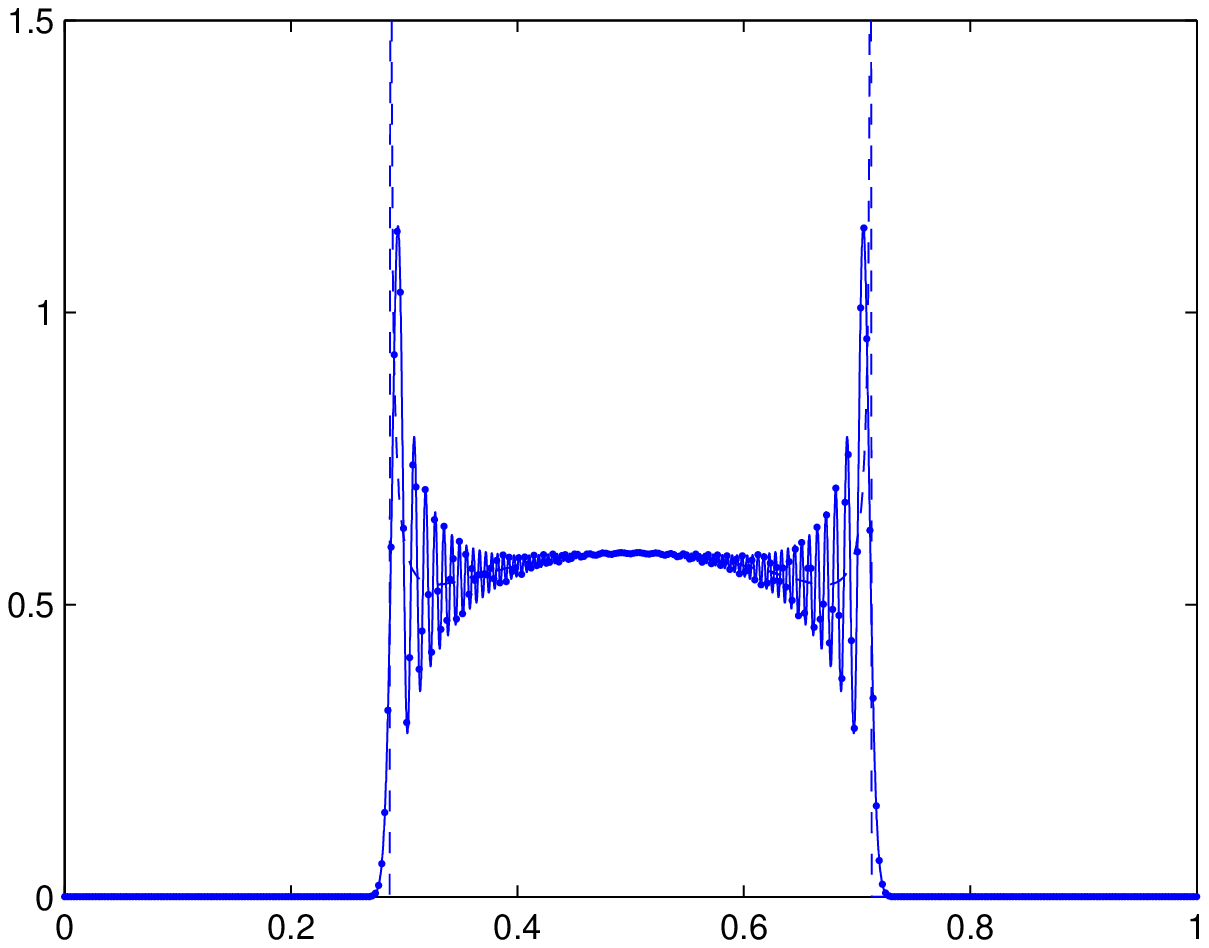}
\includegraphics[scale=0.40]{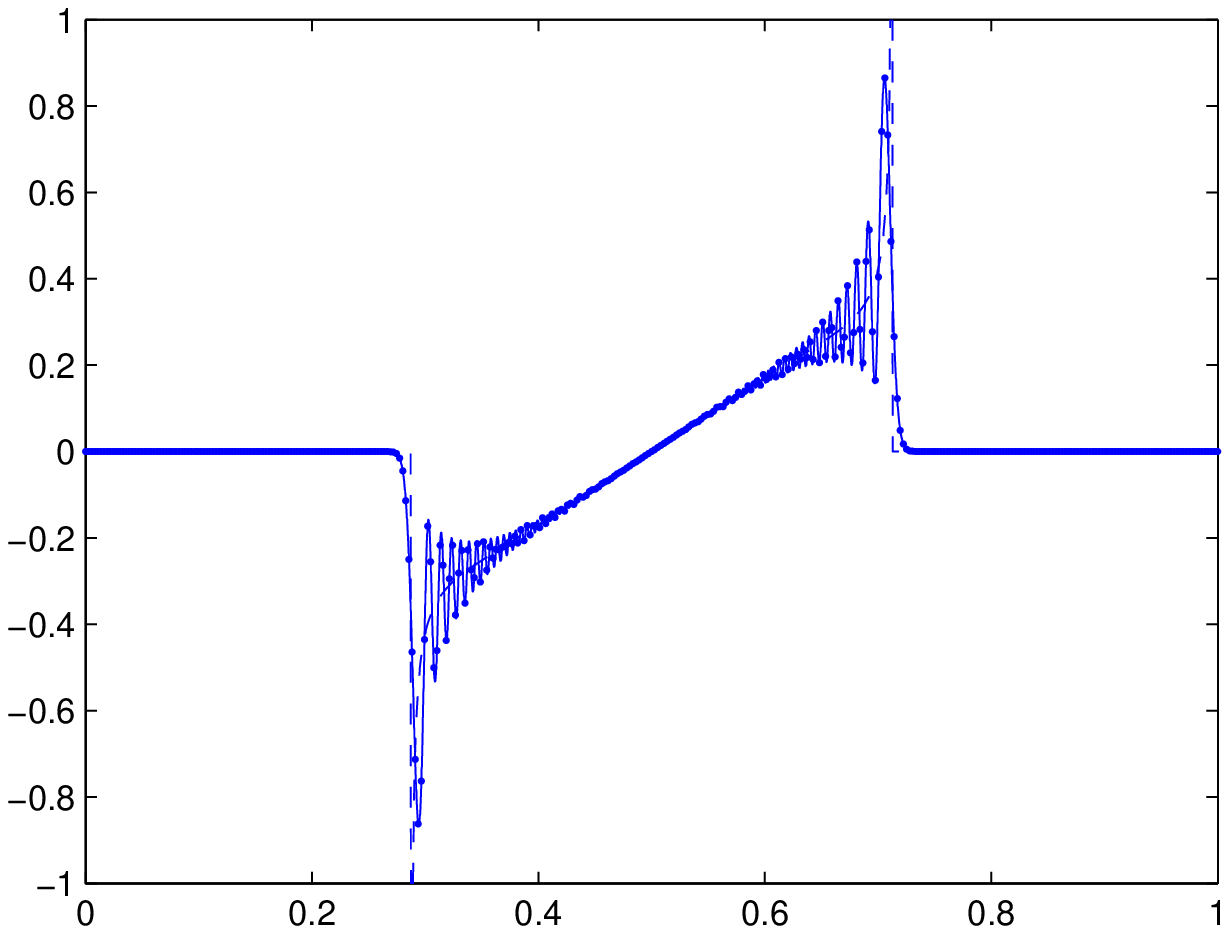}
\par \end{centering}
\begin{centering}
\includegraphics[scale=0.40]{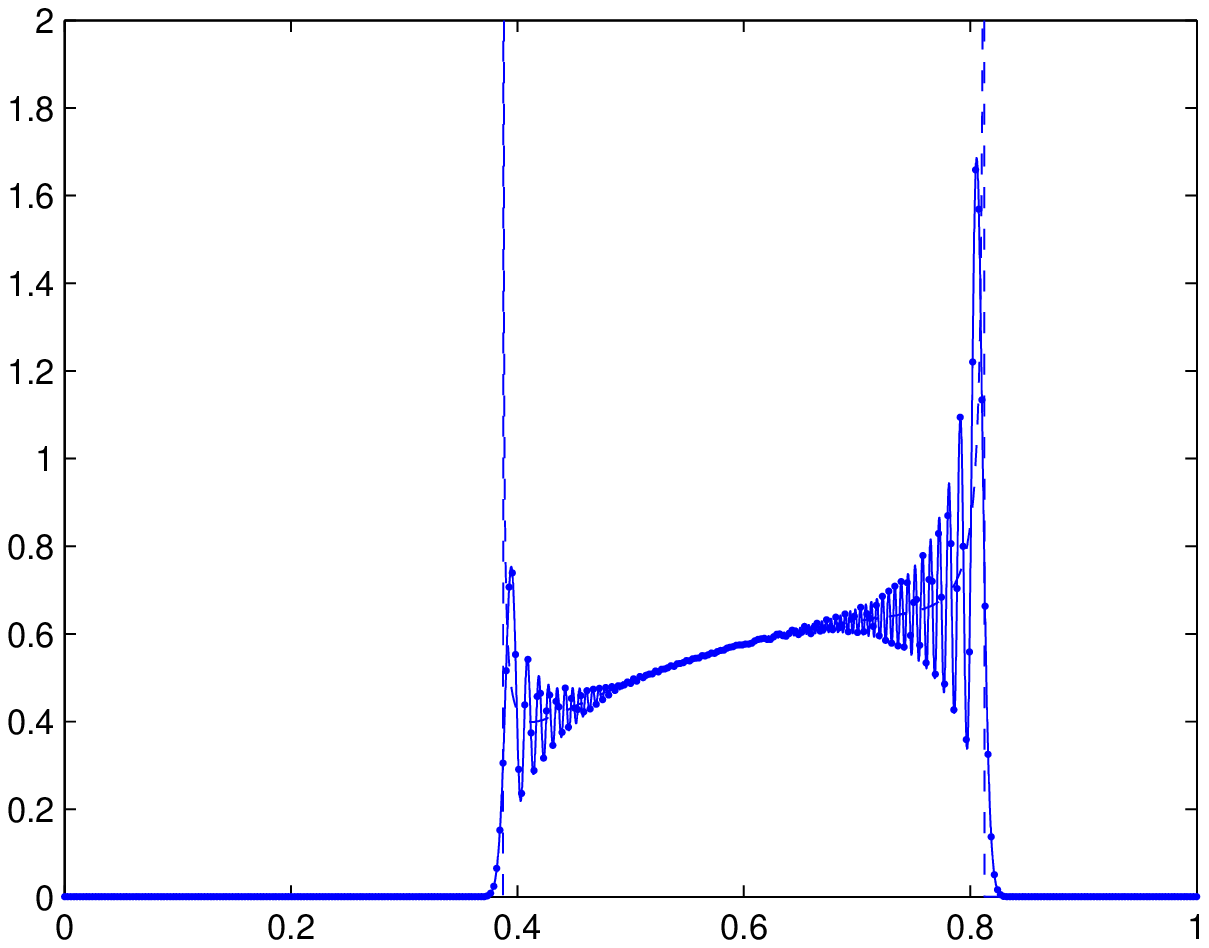}
\includegraphics[scale=0.40]{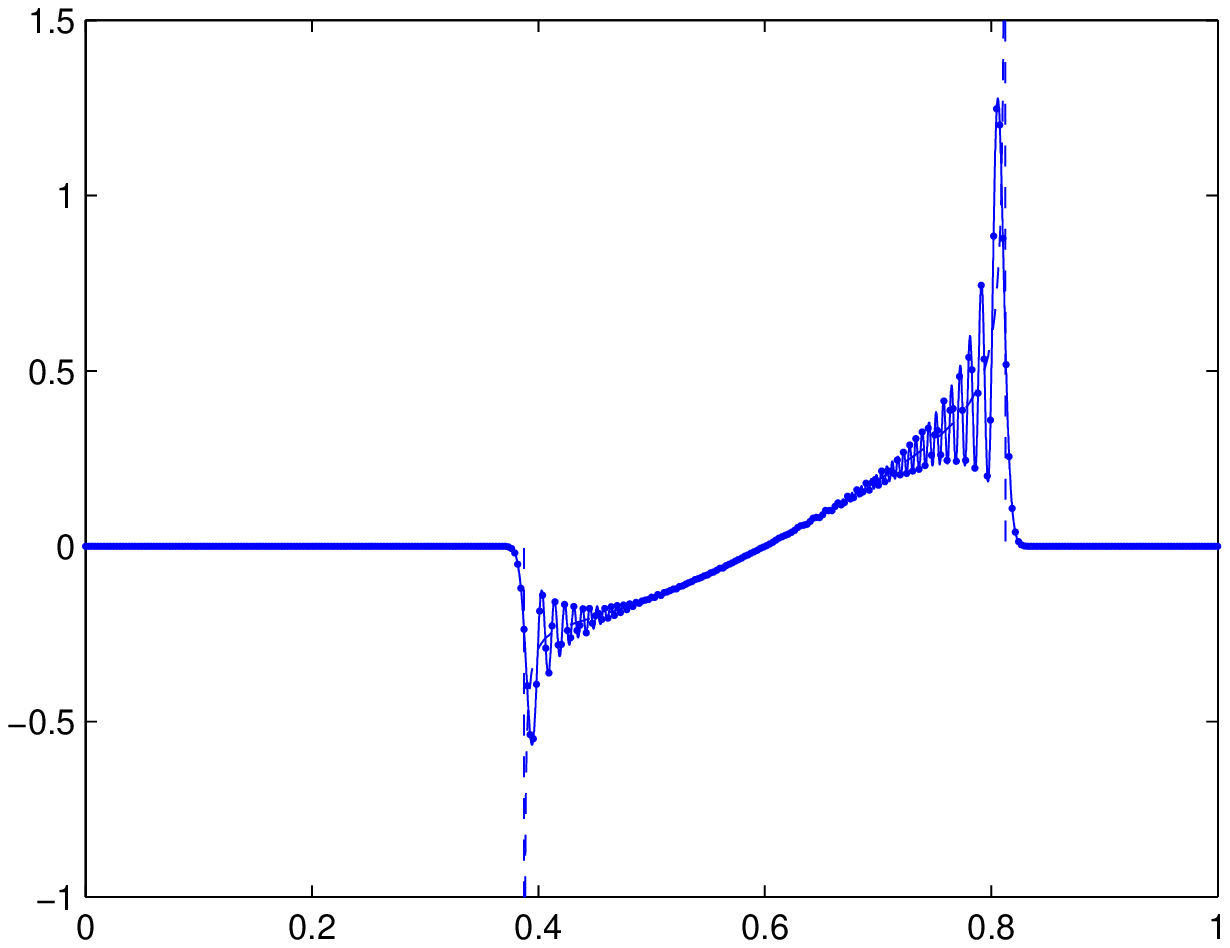}
\par \end{centering}
\caption{$\varepsilon=\frac{1}{2048}$. First row: position density and flux density of $\varphi^{\e}$; second row: position density and current density of $\psi^{\e}$. }
\end{figure}

Next, we take a more generic potential ensuring in a nontrivial coupling between the two sub-systems, namely $V(x,y)=\frac{1}{2}(x+y)^2$, i.e., a harmonic coupling. 
We again want to test whether $\varepsilon$-independent $\Delta t$ can be taken to correctly capture the behavior of physical observables. 
We solve the TDSCF equations with resolved spatial grids, which means $\Delta x=O(\varepsilon)$. The numerical solutions with $\Delta t=O(\varepsilon)$ are used as the 
reference solutions. For $\varepsilon=\frac{1}{256}$, $\frac1{512}$, $\frac1{1024}$, $\frac1{2048}$, $\frac1{4096}$, we fix $\Delta t=0.005$, and compute till $T=0.54$. 
The $l^2$ norm of the error for the 
wave functions and the error for the position densities is calculated. 
We see in Figure 7 that the former increases as $\varepsilon \rightarrow 0_{+}$, but the error in the physical observables does not change noticeably.

\begin{figure}
\begin{centering}
\includegraphics[scale=0.6]{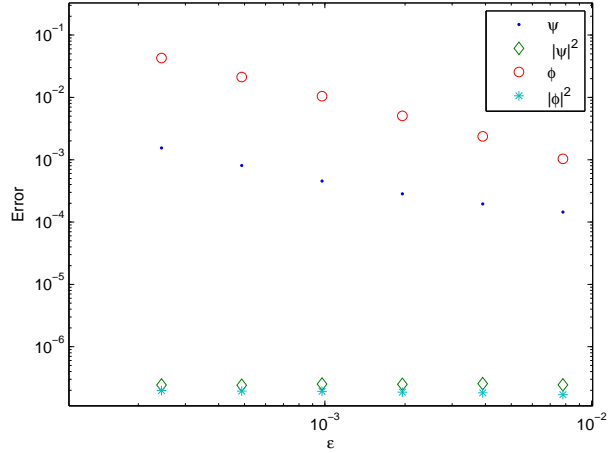} \par
\caption{Fix $\Delta$ t=0.005. For $\varepsilon=\frac{1}{256}$, $\frac1{512}$, $\frac1{1024}$, $\frac1{2048}$, $\frac1{4096}$, $\Delta x=\frac{\varepsilon}{8}$, respectively. 
The reference solution is computed with the same $\Delta x$, but $\Delta t=\frac{0.54\varepsilon}{4}$.}
\end{centering}
\end{figure}
\medskip{}

\textbf{Example 3.} In this example, we want to test the convergence in the spatial grid $\Delta x$ and in the time step $\Delta t$. According to the previous analysis, 
the spatial oscillations of wavelength $O(\e)$ need to be resolved. On the other hand, if the time oscillation with frequency $O(1/\e)$ are resolved, one gets accurate 
approximation even of the wave functions itself (not only quadratic quantities of it). Unresolved time steps of order $O(1)$ can still give correct physical observable densities.
 More specifically, one expects second order convergence with respect time in both wave functions (and in the physical observables), and spectral convergence in the 
respective spatial variable.

Assume $x,y\in[-\pi,\pi]$ and let $V(x,y)=\frac12(x+y)^2$. The initial conditions are of the WKB form,
$$
\psi^{\e}_{\rm in}(x)=e^{-5(x+0.1)^2}e^{i\sin{x}/\e}, \quad \varphi^{\e}_{\rm in}(y)=e^{-5(y-0.1)^2}e^{i\cos{y}/\varepsilon}.
$$

To test the spatial convergence, we take $\varepsilon=\frac{1}{256}$, and the reference solution is computed by well resolved mesh $\Delta x= \frac{2\pi\varepsilon}{64}$, $\Delta t=\frac{0.4 \varepsilon}{16}$ until $T=0.4$. Then, we compute with the same time step, but with difference spatial grids. The results are illustrated in Figure 8. We observe that, when $\Delta x=O(\varepsilon)$, the error decays quickly to be negligibly small as $\Delta x$ decreases. However, when the spatial grids do not well resolve the $\varepsilon$-scale, the method would actually give solutions with $O(1)$ error. 

\begin{figure}
\begin{centering}
\includegraphics[scale=0.6]{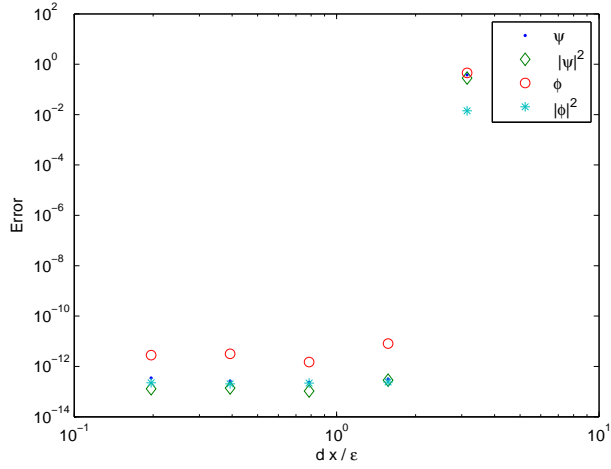} \par
\caption{Fix $\varepsilon=\frac{1}{256}$ and $\Delta t=\frac{0.4 \varepsilon}{16}$. Take $\Delta x= \frac{2\pi\varepsilon}{32}$, $\frac{2\pi\varepsilon}{16}$, 
$\frac{2\pi\varepsilon}{8}$, $\frac{2\pi\varepsilon}{4}$, $\frac{2\pi\varepsilon}{2}$ and $\frac{2\pi\varepsilon}{1}$ respectively. 
The reference solution is computed with the same $\Delta t$, but $\Delta x=\frac{2\pi\varepsilon}{64}$.}
\end{centering}
\end{figure}

At last, to test the convergence in time, we take $\varepsilon=\frac{1}{1024}$, and the reference solution is computed through a well resolved mesh with 
$\Delta x= \frac{2\pi\varepsilon}{16}$, $\Delta t=\frac{0.4}{8192}$ till $T=0.4$. 
Then, we compute with the same spatial grids, but with different time steps. 
The results are illustrated in the Figure 9. We observe that the method is stable even if $\Delta t \gg \varepsilon$. Moreover 
we get second order convergence in the wave functions as well as in the physical observable densities. 

\begin{figure}
\begin{centering}
\includegraphics[scale=0.6]{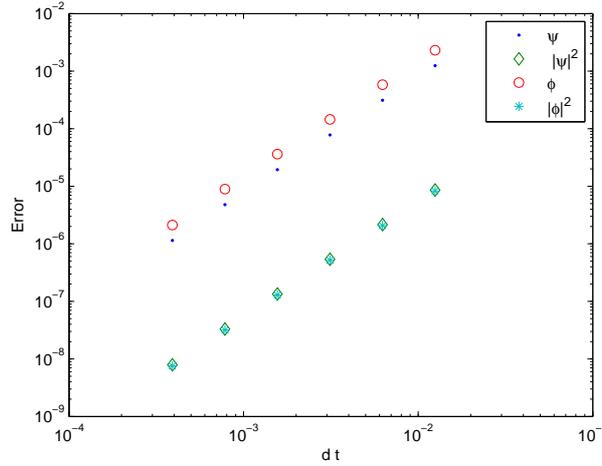} \par
\caption{Fix $\varepsilon=\frac{1}{1024}$ and $\Delta x=\frac{2 \pi}{16}$. 
Take $\Delta t= \frac{0.4}{32}$, $\frac{0.4}{64}$, $\frac{0.4}{128}$, $\frac{0.4}{256}$, $\frac{0.4}{512}$ and $\frac{0.4}{1024}$, respectively. 
The reference solution is computed with the same $\Delta x$, but $\Delta t=\frac{0.4}{8192}$.}
\end{centering}
\end{figure}

\medskip{}


\subsection{SVSP2 method for the Ehrenfest equations}

Now we solve the Ehrenfest equations (\ref{eq:ehrenfest2}) by the SVSP2 method. Assume $x\in[-\pi,\pi]$, and assume periodic boundary conditions for the electronic wave equation.
\medskip{}

\textbf{Example 4.} In this example, we want to test if $\delta$-independent time steps can be taken to capture correct physical observables and the convergence in the time step which is expected to be of the second order. The potential is again $V(x,y)=\frac{1}{2}(x+y)^2$ and the initial conditions are chosen to be
$$
\psi^{\delta}(x,0)=e^{-5(x+0.1)^2}e^{i\sin{x}/\delta}, \quad y(0)=0,\quad \eta(0)=0.1.
$$

First, we test whether $\delta$-independent $\Delta t$ can be taken to capture the correct physical observables. 
We solve the equations with resolved spatial grids, which means $\Delta x=O(\delta)$. The numerical solutions with $\Delta t=O(\delta)$ 
are used as the reference solutions. For $\delta=1/256$, $1/512$, $1/1024$, $1/2048$, $1/4096$, we fix $\Delta t=\frac{0.4}{64}$, 
and compute until $T=0.4$. The $l^2$ norm of the error in wave functions, the error in position densities, and the error in the coordinates of the nucleus are calculated. 
We see in Figure 10 that the error in the wave functions increases as $\delta \rightarrow 0_{+}$, but the errors in physical observables and in the classical coordinates do not change notably.

\begin{figure}
\begin{centering}
\includegraphics[scale=0.6]{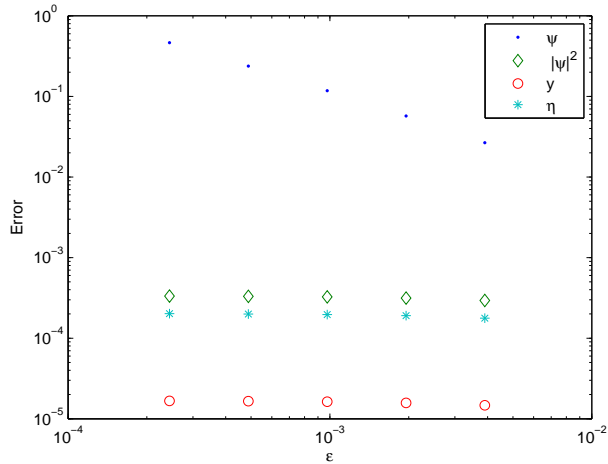} \par
\caption{Fix $\Delta t=\frac{0.4}{64}$. For $\delta=\frac1{256}$, $\frac1{512}$, $\frac1{1024}$, $\frac1{2048}$, $\frac1{4096}$, 
$\Delta x=2\pi\varepsilon/16$, respectively. The reference solution is computed with the same $\Delta x$, but $\Delta t=\frac{\delta}{10}$.}
\end{centering}
\end{figure}

Next, we test the convergences with respect to the time step in the wave function, the physical observables and the classical coordinates. 
We take $\delta=\frac{1}{1024}$, and the reference solution is computed by well resolved mesh $\Delta x= \frac{2\pi\varepsilon}{16}$, $\Delta t=\frac{0.4}{8192}$ till $T=0.4$. 
Then, we compute with the same spatial grids, but with difference time steps. The results are illustrated in the Figure 11. 
We observe that, the method is stable even if $\Delta t \gg \varepsilon$, and clearly, 
we get second order convergence in the wave functions, the physical observable densities and the classical coordinates. 

\begin{figure}
\begin{centering}
\includegraphics[scale=0.6]{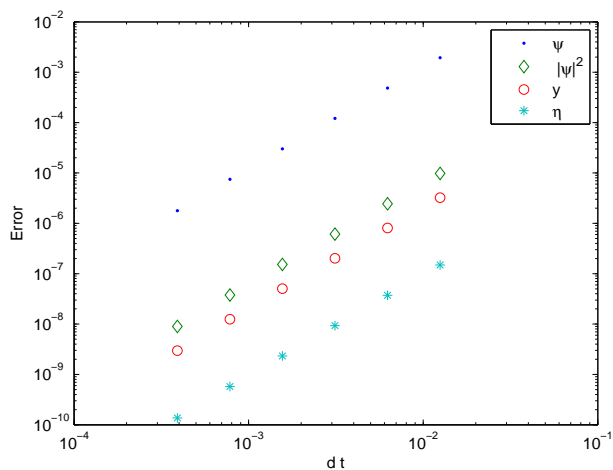} \par
\caption{Fix $\delta=\frac{1}{1024}$ and $\Delta x=\frac{2 \pi}{16}$. Take $\Delta t= \frac{0.4}{32}$, $\frac{0.4}{64}$, $\frac{0.4}{128}$, $\frac{0.4}{256}$, $\frac{0.4}{512}$ and $\frac{0.4}{1024}$, 
respectively. The reference solution is computed with the same $\Delta x$, but $\Delta t=\frac{0.4}{8192}$.}
\end{centering}
\end{figure}



\begin{thebibliography}{amsplain}

\bibitem{AMS} G. Aki, P. Markowich, and C. Sparber, {\it Classical limit for semi-relativistic Hartree systems}, J. Math. Phys. {\bf 49} (2008), no.10, 102110, 10 pp.

\bibitem{NLS}W. Bao, S. Jin, and P. A. Markowich, {\it Numerical study of time-splitting spectral discretizations of nonlinear Schr\"odinger equations in the semiclassical regimes}, SIAM J. Sci. Comput. {\bf 25} (2003), no.1, 27--64.

\bibitem{TS}W. Bao, S. Jin, and P. A. Markowich, {\it On Time-Splitting Spectral Approximations for the Schr\"odinger Equation
in the Semi-classical Regime}, J. Comput. Phys. {\bf 175} (2002), no.2, 487--524.

\bibitem{BS}  F. Bornemann and C. Sch\"utte, {\it On the Singular Limit of the Quantum-Classical Molecular Dynamics Model}, SIAM J. Appl. Math. {\bf 59} (1999) 1208--1224.

\bibitem{BNS} F. Bornemann, P. Nettesheim, and C. Sch\"utte, {\it Quantum-classical Molecular Dynamics as an Approximation to Full Quantum Dynamics}, 
J. Chem. Phys. {\bf 105} (1996), 1074--1083.

\bibitem{Caz} T. Cazenave, {Semilinear Schr\"odinger equations}, Courant Lecture Notes in Mathematics vo. 10, New York University, 2003.

\bibitem{Dr} K. Drukker, {\it Basics of Surface Hopping in Mixed Quantum/Classical Simulations}, J. Comput. Phys. {\bf 153} (1999), 225--272.

\bibitem{GaMa} I. Gasser and P. A. Markowich, {\it Quantum hydrodynamics, Wigner transforms and the classical limit}, Asymptotic Analysis {\bf 14} (1997), no. 2, 97--116.

\bibitem{GMMP} P. G\'{e}rard, P. Markowich, N. Mauser, and F. Poupaud, {\it Homogenization Limits and Wigner transforms}, Comm. Pure Appl. Math. {\bf 50} (1997), 323--379.

\bibitem{GBR} R. B. Gerber, V. Buch, and M. A. Ratner, {\it Time-dependent self-consistent field approximation for intramolecular energy transfer. I. 
Formulation and application to dissociation of van der Waals molecules}, J. Chem. Phys. {\bf 77} (1982), 3022--3030.

\bibitem{GrRa} R. B. Gerber and M. A. Ratner,{\it Mean-field models for molecular states and dynamics New developments}, J. Phys. Chem. {\bf 92} (1988), 3252--3260.

\bibitem{OdeSp} E. Hairer, C. Lubich, and G. Wanner, { Geometric numerical integration: structure-preserving algorithms for ordinary differential equations}, vo. 31, Springer, 2006.

\bibitem{H} J. Hinze, {\it MC-SCF. I. The multi-configurational self-consistent-field method}, J. Chem. Phys. {\bf 59} (1973), 6424--6432.

\bibitem{Cau1}S. Jin, C. D. Levermore, and D. W. McLaughlin, {\it The behavior of solutions of the NLS equation in the semi-classical limit, in Singular Limits of Dispersive Waves}, in {\it Singular limits of dispersive waves}, Springer US (1994), 235--255.

\bibitem{JMS} S. Jin, P. Markowich, and C. Sparber, {\it Mathematical and computational methods for semi-classical Schr\"odinger equations}, Acta Num. {\bf 20} (2011), 211--289.

\bibitem{SL-TS} S. Jin and Z. Zhou, {\it A semi-Lagrangian time splitting method for the Schr$\ddot{\rm{o}}$dinger equation with vector potentials}, Comm. Inf. Syst. {\bf 13} (2013), no. 3, 247--289.

\bibitem{Cau2} P. A. Markowich, P. Pietra, and C. Pohl, {\it Numerical approximation of quadratic observables of Schr\"odinger type equations in the semi-classical limit}, Numer. Math. {\bf81}, (1999), 595--630.

\bibitem{Kl} C. Klein, \textit{Fourth order time-stepping for low dispersion Korteweg-de Vries and nonlinear Schr\"odinger equations.}
Electronic Trans. Num. Anal. \textbf{29} (2008), 116--135.

\bibitem{ZNN} Z. Kotler, E. Neria, and A. Nitzan, {\it Multiconfiguration time-dependent self-consistent field approximations in the numerical-solution of quantum dynamic problems}, 
Comput. Phys. Comm. {\bf 63} (1991), 243--258.

\bibitem{LiPa} P.-L. Lions and T. Paul, {\it Sur les measures de Wigner}, Rev. Math. Iberoamericana {\bf 9} (1993), 553--618.

\bibitem{MaMa} P. Markowich and N. Mauser, {\it The classical limit of a self-consistent quantum-Vlasov equation in 3-D}, Math. Models Methods Appl. Sci. {\bf 3} (1993), no. 1, 109--124.

\bibitem{SpTe} H.~Spohn and S.~Teufel, {\it Adiabatic Decoupling and Time-Dependent Born Oppenheimer Theory}, Comm. Math. Phys. {\bf 224} (2001), issue 1, 113--132.

\bibitem{SMM} C. Sparber, P. A. Markowich, and N. J. Mauser, {\it Wigner functions versus WKB-methods in multivalued geometrical optics}, Asymptotic Analysis {\bf 33} (2003), no. 2, 153--187. 

\bibitem{SuMi} X. Sun and W. H. Miller, {\it Mixed semi-classical classical approaches to the dynamics of complex molecular systems}, J. Chem. Phys. {\bf 106} (1997), no.3, 916--927.

\bibitem{Wi} E. Wigner, {\it On the Quantum Correction for the Thermodynamic Equilibrium}, Phys. Rev. {\bf 40} (1932), 749--759.

\end{thebibliography}
\end{document}